\documentclass[11pt]{amsart}
\makeatletter
\@namedef{subjclassname@2020}{%
  \textup{2020} Mathematics Subject Classification}
\makeatother
\usepackage[utf8]{inputenc}
\usepackage[T1]{fontenc}
\usepackage{latexsym, amsmath, amssymb}
\usepackage{amsmath}
\usepackage{amsfonts}
\usepackage{amssymb}
\usepackage{amscd}
\usepackage{exscale}
\usepackage{relsize}
\usepackage{amsthm}
\usepackage{float}
\restylefloat{table}

\usepackage{verbatim}
\usepackage[dvipsnames]{xcolor}
\usepackage{xypic}
\usepackage{tikz}
\usetikzlibrary{matrix,shapes,backgrounds,fit}
\usepackage{tikz-cd}
\usepackage{latexsym}
\usepackage{todonotes}
\usepackage{mathrsfs}
\theoremstyle{plain}
\newtheorem*{theoremint}{Main Theorem}
\newtheorem{theorem}{Theorem}[section]
\newtheorem{proposition}[theorem]{Proposition}
\newtheorem{lemma}[theorem]{Lemma}
\newtheorem{corollary}[theorem]{Corollary}
\theoremstyle{definition}
\newtheorem{definition}[theorem]{Definition}

\newtheorem{remark}[theorem]{Remark}

\numberwithin{equation}{section}

\DeclareMathOperator{\Ext}{Ext}

\DeclareMathOperator{\rank}{rk}

\newcommand{\arr}{\longrightarrow}
\newcommand{\rk}{\mathrm{rank}}

\newcommand{\ppp}{\mathbb{P}^1\times\mathbb{P}^1\times\mathbb{P}^1}

\def\Ext{{\rm{Ext\,}}}

\newcommand{\zz}{{\mathbb Z}}

\newcommand{\pp}{{\mathbb P}}

\newcommand{\Pic}{\mathrm{Pic}}

\newcommand\sA{{\mathcal A}}

\newcommand\sD{{\mathcal D}}
\newcommand\sE{{\mathcal E}}
\newcommand\sF{{\mathcal F}}
\newcommand\sG{{\mathcal G}}
\newcommand\sH{{\mathcal H}}
\newcommand\sI{{\mathcal I}}

\newcommand\sL{{\mathcal L}}
\newcommand\sM{{\mathcal M}}
\newcommand\sN{{\mathcal N}}
\newcommand\sO{{\mathcal O}}

\newcommand\sU{{\mathcal U}}

\newcommand\sY{{\mathcal Y}}

\newcommand\fE{{\mathfrak E}}
\newcommand\fF{{\mathfrak F}}

\newcommand\fH{{\mathfrak H}}
\newcommand\fI{{\mathfrak I}}

\newcommand\fP{{\mathfrak P}}

\newcommand\fX{{\mathfrak X}}

\def\pee#1{\hbox{$ {\mathbb P}^{#1}$}}

  \def \tab#1{\kern #1 truein}
  {\left\lbrace\begin{array}{@{}l@{}}}%
  {\end{array}\right.}

\title[Stratification of moduli spaces of instantons via 't Hooft bundles]{Stratification of moduli spaces of instantons on the Segre product of three lines via 't Hooft bundles}
\author{V.Antonelli, F. Malaspina}
\address{Politecnico di Torino, Corso Duca degli Abruzzi 24, 10129 Torino, Italy}
\email{vincenzo.antonelli@polito.it}

\address{Politecnico di Torino, Corso Duca degli Abruzzi 24, 10129 Torino, Italy}
\email{francesco.malaspina@polito.it}

\thanks{Both the authors are members of GNSAGA group of INDAM. The first author acknowledges support from the GNSAGA project: ``Conic bundle structures and vector bundles in the birational geometry of Fano varieties" (CUP E53C25002010001)}
\keywords{Instanton bundles, 't Hooft bundles, moduli spaces, Fano threefolds}
\subjclass[2020]{Primary: 14J60. Secondary: 14F06, 14J45, 14D21}
\begin{document}
\begin{abstract}
Let $X$ be the Segre product of three projective lines. For a fixed effective divisor $D$ on $X$, we introduce the notions of $D$-'t Hooft, $(D_i,D_j)$-special and $D$-sectional special bundle. The varieties parameterizing these bundles yield a natural stratification of the moduli space of stable instanton bundles with fixed Chern classes. After characterizing the curves associated with these bundles via Serre correspondence, we describe the corresponding Hilbert schemes. Using this description, we analyze the moduli spaces of $h_i$-'t Hooft bundles and the smaller strata of $(h_i,h_j)$-special and $(h_i)$-sectional special bundles. Finally, we provide a detailed study of the low-charge cases. 
\end{abstract}
\maketitle

\section{Introduction}
We work over the field $\mathbb{C}$ of complex numbers. Let 
$X \subset \mathbb{P}^N$ be a smooth polarized projective variety of dimension $n$. In the study of the geometry of $X$ and of its 
subvarieties, vector bundles on $X$ have proved to be a powerful and 
flexible tool. It is therefore natural to focus on distinguished 
classes of vector bundles reflecting special geometric or cohomological 
features of the underlying variety.

Among such classes, instanton bundles have attracted sustained interest 
in algebraic geometry. Their origin lies in mathematical physics: they 
were first introduced on the projective space $\mathbb{P}^3$ as the 
algebraic counterparts of solutions of the Yang-Mills equations on the 
four-sphere $S^4$. In the original definition \cite{ADHM,AW}, an instanton bundle $\sE$ on 
$\mathbb{P}^3$ is a stable rank-two vector bundle with $c_1(\sE)=0$ and 
satisfying
\[
H^1(\sE(-2))=0.
\]
The second Chern class $c_2(\sE)$ is called the \emph{charge} of the 
instanton.

Motivated by the rich geometry of the moduli spaces of instantons on 
$\mathbb{P}^3$, this notion has been generalized to other projective 
varieties, with the aim of understanding vector bundles on threefolds 
and, more generally, on higher-dimensional varieties. Since 
$\mathbb{P}^3$ is the only Fano threefold of index $4$, it is natural to 
seek analogues on other Fano threefolds of Picard number one 
(cf. \cite{Fa2,Kuz}). The first definitions of instanton bundles on Fano 
varieties of higher Picard rank were proposed in 
\cite{MMP,AnMa}. Subsequently, the notion was extended to arbitrary 
polarized threefolds \cite{AnMa2} and even to arbitrary projective 
schemes \cite{AnCa}. Although the most general definitions no longer 
retain a direct physical interpretation, instanton bundles continue to 
play a central role in the study of moduli spaces of stable vector 
bundles.

Within the class of instanton bundles, the so-called 
\emph{’t Hooft instanton bundles} occupy a distinguished position. The 
terminology originates in the seminal work of ’t Hooft \cite{Hooft}, 
where explicit multi-instanton solutions of the anti-self-dual 
Yang-Mills equations were constructed by superposition of basic 
instantons. Via the Ward correspondence and the ADHM construction, these 
solutions correspond to algebraic instanton bundles $\sE$ on 
$\mathbb{P}^3$ such that the twist $\sE(1)$ admits a non-zero global 
section. On $\mathbb{P}^3$, ’t Hooft instanton bundles form some of the 
most explicit families of instantons: they admit monadic descriptions 
and determine irreducible components and boundary strata of the 
corresponding moduli spaces (cf. \cite{BeFr,BT,Ha2,HN}). Results on the total moduli space of instantons can be found in \cite{CO,JV,T1,T2}

These features suggest that suitable analogues of ’t Hooft instanton 
bundles should play a key role in the study of moduli spaces of stable 
bundles on other varieties. A first step in this direction was carried 
out in \cite{AMMP}, where the authors introduced the notion of 
$D$-’t Hooft bundles on the complete flag threefold 
$F$ of point-line in $\mathbb{P}^2$, and described their moduli 
spaces, revealing a behavior markedly different from the classical case 
of $\mathbb{P}^3$ (cf. \cite[Theorem B]{AMMP}).

Apart from $F$, there is only one other Fano threefold $X$ with 
$K_X=-2h$ and Picard rank greater than one, namely the Segre threefold $X = \mathbb{P}^1 \times \mathbb{P}^1 \times \mathbb{P}^1$
embedded in $\pp^7$ by the very ample line bundle $\sO_X(h)$. Instanton bundles on $X$ were first studied in \cite{AnMa}. Let us recall that an instanton bundle $\sE$ on $X$ is a $\mu$-semistable, rank two vector bundle with $c_1(\sE)=0$ satisfying $$h^0(\sE)=h^1(\sE(-h))=0.$$
The aim of this paper is to introduce and study ’t Hooft instanton 
bundles on $X$ and to analyze the geometry of their moduli spaces. The multigraded geometry and the higher Picard rank of $X$ lead to a 
richer notion of instanton and require new techniques for their 
construction and classification. Given an effective divisor $D$ on $X$, 
we define \emph{$D$-’t Hooft bundles} as instanton bundles $\sE$ such 
that $H^0(\sE(D)) \neq 0$. If $h_1,h_2,h_3$ denote the generators of 
$\Pic(X)$ given by the pullbacks of $\mathcal{O}_{\mathbb{P}^1}(1)$ from 
each factor, we focus in particular on $h_i$-’t Hooft bundles, which 
retain several key features of the classical case.

We first characterize the curves $Y_i$ that arise as zero loci of 
sections $s \in H^0(\sE(h_i))$ and describe their Hilbert schemes. We 
show that $h^0(\sE(h_i)) \leq 2$ for every instanton bundle $\sE$ on $X$, independently of the charge, and we introduce two different notions of speciality: 
$h_i$-sectional special bundles, for which $h^0(\sE(h_i))=2$, and 
$(h_i,h_j)$-special bundles, for which both $\sE(h_i)$ and $\sE(h_j)$ 
admit non-zero sections. We analyze how these properties are reflected, 
via Serre correspondence, in the geometry of the associated curves, and 
we study the corresponding strata inside the moduli space of instanton 
bundles. Let us indicate by $\sY$ the class $k_1h_2h_3+k_2h_1h_3+k_3h_1h_2$ in $A^2(X)$ and let us denote by $MI(\sY)$ the moduli spaces of instanton bundles having $c_2=\sY$, i.e. the open subset of $\mu$-stable instanton bundles inside the Maruyama moduli space $M_X(2;0,\sY)$ of rank two $\mu$-stable bundles with $c_1=0$ and $c_2=\sY$. Furthermore, we denote by $MI^{h_i}(\sY)$, $MI^{(h_i,h_j)}(\sY)$ and $MI^{(h_i,h_i)}(\sY)$ respectively the moduli space of $h_i$-'t Hooft, $(h_i,h_j)$-special and $h_i$-sectional special bundles. We collect the results from Theorem \ref{tModuliSectional}, \ref{tModuliSpecial} and \ref{tModulitHooft} in the following theorem.

\begin{theoremint}
The moduli space $MI^{(h_i,h_i)}(\sY)$ is smooth and irreducible of  dimension $2(k_j+k_w)+1$ with $i\neq j \neq w \neq i$ and it is non-empty if and only if $k_i=0$. The moduli space $MI^{(h_i,h_j)}(\sY)$ is irreducible and smooth outside of $h_i$-sectional special and $h_j$-sectional special bundles. The maximal stratum $MI^{h_i}(\sY)$ is smooth outside $MI^{(h_i,h_i)}$ and consists of 
\[
\mathlarger{\mathlarger{\sum}}_{{\substack{ 0\le \alpha \le k_j \\ 0 \le \beta \le k_w \\ 1 \le \alpha + \beta \le k_i}}}{\binom{k_i-1}{\alpha+\beta-1}}
\]
irreducible components if $k_i\neq 0$ and it is irreducible if $k_i=0$. Moreover the reducible union $MI^{h_i}(\sY) \cup MI^{h_j}(\sY)$ is also singular along $MI^{(h_i,h_j)}(\sY)$.
\end{theoremint}

We conclude by outlining the structure of the paper. 
In Section 2 we describe certain distinguished surfaces and curves in 
$X$. In Section 3 we recall the definition of instanton bundles and 
collect auxiliary results on vector bundles. In Section 4 we study 
$h_i$-’t Hooft bundles, their associated curves, and the relevant 
Hilbert schemes. In Section 5 we introduce and analyze the two notions 
of speciality. In Section 6 we describe the moduli spaces of 
$h_i$-’t Hooft, $(h_i,h_j)$-special, and $h_i$-sectional special 
bundles, discussing smoothness and irreducible components. Finally, in 
Section 7 we treat the low-charge cases, in which $h_i$-’t Hooft 
bundles exhaust the full moduli space of instanton bundles.

\section{Some geometry of $\ppp$}
Let $V_1, V_2, V_3$ be three $2$-dimensional complex vector spaces with coordinates $[x_{1i}], [x_{2j}], [x_{3w}]$ respectively with $i,j,w\in \{1,2\}$. Let $X\cong \mathbb P (V_1) \times \mathbb P (V_2) \times \mathbb P (V_3)$ and then it is embedded into $\mathbb P^7\cong \mathbb P(V)$ by the Segre map where $V=V_1 \otimes V_2 \otimes V_3$.

The intersection ring $A(X)$ is isomorphic to $A(\mathbb P^1) \otimes A(\mathbb P^1) \otimes A(\mathbb P^1)$ and so we have
$$A(X) \cong \mathbb Z[h_1, h_2, h_3]/(h_1^2, h_2^2, h_3^2).$$
We may identify $A^1(X)\cong \mathbb Z^{\oplus 3}$ by $a_1h_1+a_2h_2+a_3h_3 \mapsto (a_1, a_2, a_3)$. Similarly we have $A^2(X) \cong \mathbb Z^{\oplus 3}$ by $k_1e_1+k_2e_2+k_3e_3\mapsto (k_1, k_2, k_3)$ where $e_1=h_2h_3, e_2=h_1h_3, e_3=h_1h_2$ and $A^3(X) \cong \mathbb Z$ by $ch_1h_2h_3 \mapsto c$.
Then $X$ is embedded into $\mathbb P^7$ by the complete linear system $h=h_1+h_2+h_3$ as a smooth subvariety of degree $6$. 

We will denote by $\pi_i:X\to \mathbb{P}(V_i)$ the projection onto the $i$-th factor and by $\pi_{ij}:X\to \mathbb{P}(V_i)\times \mathbb{P}(V_j)$ the projection onto the $(i,j)$-th quadric.

Now we describe some notable surfaces contained in $X$.

\begin{lemma}\label{lQuadrics}
Let $S$ be a surface in the linear system $|h_i|$. Then $S$ is isomorphic to the polarized surface $(\pp^1\times\pp^1,\sO_{\pp^1\times\pp^1}(1,1))$. Moreover, the restriction map
\[
\phi:\Pic(X)\cong \zz^{\oplus3}\langle h_1,h_2,h_3\rangle\to \Pic(S)\cong\zz^{\oplus 2} \langle l,m\rangle
\]
is completely determined by 
\[
\phi(h_i)=0, \quad \phi(h_j)=l \quad \text{and} \quad \phi(h_k)=m,
\]
with $i\neq j \neq k \neq i$.
\end{lemma}
\begin{proof}
It is straightforward since the intersection product between $h_i$ and $h^2$ is $h_ih=2$. The final part concerning the Picard groups can be obtained by restricting each generator of $\Pic(X)$ to $S$ and by the linearity of $\phi$.
\end{proof}
Notice that two different quadrics in the same linear system are always disjoint, while quadrics in different linear systems intersect along a line.
\begin{lemma}\label{lQuartics}
Let $S$ be a surface in the linear system $|h_i+h_j|$ with $i\neq j$. Then one of the following cases holds:
\begin{itemize}
\item $S$ is smooth and irreducible and is isomorphic to $(\pp^1\times\pp^1,\sO_{\pp^1\times \pp^1}(2\ell+m))$;
\item $S$ is the reducible union of two quadrics $Q_i$ and $Q_j$ in the linear systems $|h_i|$ and $|h_j|$.
\end{itemize}
Moreover, in the irreducible case the restriction map
\[
\phi:\Pic(X)\cong \zz^{\oplus3}\langle h_1,h_2,h_3\rangle\to \Pic(S)\cong\zz^{\oplus 2} \langle l,m\rangle
\]
is completely determined by
\[
\phi(h_i)=l \quad \phi(h_j)=l \quad \text{and} \quad \phi(h_w)=m
\]
with $i \neq j \neq w \neq i$.
\end{lemma}
\begin{proof}
We only prove the irreducible case, since the reducible one is a direct consequence of Lemma \ref{lQuadrics}. Without loss of generality, suppose that $S\in|h_1+h_2|$ is irreducible. Then $S$ is the pullback of a smooth conic $C$ in $\pp(V_1)\times \pp(V_2)$ via the projection $\pi_{12}$. In particular $S$ is isomorphic to the product $C\times\pp(V_3)$, thus $S\cong (\pp^1\times\pp^1,\sO_{\pp^1\times \pp^1}(2\ell+m))$. By computing the intersections $h_i(h_1+h_2)$, the final part of the statement follows directly from the equality $\phi(h)=2l+m$ and the linearity of $\phi$.
\end{proof}

\begin{lemma}\label{lRationalCurvesGeneral}
Every class $a_1e_1+a_2e_2+a_3e_3$ in $A^2(X)$ has a representative which is a smooth rational curve.
\end{lemma}
\begin{proof}
It is enough to consider $\varphi_i:\pp^1\to\pp^1$ a general morphism of degree $a_i$. The map $\varphi=(\varphi_1,\varphi_2,\varphi_3):\pp^1\to X$ gives a smooth rational curve representing the class $a_1e_1+a_2e_2+a_3e_3$.
\end{proof}
\begin{proposition}\label{pCompleteIntersection}
Let $C\subset X$ be a reduced connected curve representing the class $ae_1+e_2$ (resp. $ae_1+e_3$) with $a \ge 0$. Then $C$ is a complete intersection curve of type $(h_3,ah_2+h_1)$ (resp. $(h_2,ah_3+h_1)$) and arithmetic genus $p_a(C)=0$. Moreover, if $C$ is an integral curve, then $C\cong \pp^1$ and its normal sheaf is $\sN_{C|X}\cong \sO_{\pp^1}\oplus \sO_{\pp^1}(2a)$.
\end{proposition}
\begin{proof}
We only prove the statement for curves representing $ae_1+e_2$, the other case being completely analogous. Let $C\subset X$ be a curve as in the hypothesis. Then $C$ projects through $\pi_3$ to a point in $\pp(V_3)$, thus $C$ is contained in a quadric $Q\in|h_3|$. Denote by $(\ell,m)$ the two generators of $\Pic(Q)$. Thanks to Lemma \ref{lQuadrics}, $C$ represents the class $\ell+am$ in $A^1(Q)$ and the adjunction formula directly yields $p_a(C)=0$. To conclude, consider the exact sequence:
\[
0\to\sO_X\bigr(h_1+ah_2-h_3\bigl)\to\sO_X\bigr(h_1+ah_2\bigl)\xrightarrow{\phi}\sO_Q(C)\to 0.
\]
Since $h^1(\sO_X(h_1+ah_2-h_3))=0$ for $a\geq 0$, the induced map $H^0(\phi)$ is surjective and therefore any curve in the linear system $|C|$ is a complete intersection. The statement about the normal bundle now follows directly. Notice also that $\omega_C\cong\sO_C(-2h_2-2h_3)$ in the integral case.
\end{proof}

Let $C$ be a complete, smooth, and rational intersection curve as described in Proposition \ref{pCompleteIntersection}. We are now interested in the structure of non-reduced curves supported on $C$. We start by describing Cohen-Macaulay double structures on $C$, which are all obtained by the Ferrand doubling technique. We will use \cite{BF,Ma} as general references. A similar study of multiple structures supported on rational curves can be found in \cite[Section 4]{AMMP} and we refer the interested reader there for complete details.

 Let us denote by $\nu_X$ the conormal bundle of the variety $X$. For a smooth rational curve $C$ on $F$, by Proposition \ref{pCompleteIntersection} we have $\nu_C \cong \sO_{\pp^1} \oplus \sO_{\pp^1}(-2a)$. Every double structure $Y_1$ on $C$ arises from a surjective map
\[
\nu_C \xrightarrow{\phi} \mathcal{L} \to 0
\]
where $\sL$ is a line bundle on $C$. Since $C$ is rational, we get $\sL \cong \sO_{\pp^1}(\alpha)$. Notice that since $\phi$ is surjective, it forces the inequality $\alpha \ge -1$. Moreover we have the following short exact sequence
\begin{equation}\label{eSeqNilpotent}
0 \rightarrow \frac{\sI_{Y_1}}{\sI_C^2} \rightarrow \frac{\sI_{C}}{\sI_C^2} \xrightarrow{\phi} \sL\cong \frac{\sI_C}{\sI_{Y_1}} \rightarrow 0.
\end{equation}

In order to study higher multiplicity extensions, let us start by focusing on the Cohen-Macaulay extensions $Y$ which are locally contained in a smooth surface. These are called \textit{primitive extensions} of $C$, according to the following definition. 

\begin{definition}\label{dPrimitive}
Let $C$ be a smooth integral curve. A \textit{primitive extension} of $C$ is a Cohen-Macaulay curve $Y$ such that $Y_{red}\cong C$ and such that $Y$ can be locally embedded in a smooth surface. Associated to $Y$ there is a canonical filtration 
\begin{equation}\label{eFiltration}
C=Y_0\subset Y_1\subset\dots Y_k=Y 
\end{equation}
where $Y_j = Y \cap C^{(j)}$ and $C^{(j)}$ is the $j$-th infinitesimal neighborhood of $C$. The integer $k+1$ is the multiplicity of $Y$. In this situation, $\sL:=\sI_{C\mid Y_1}$ is a line bundle on $C$. It is called \textit{the type} of $Y$.
\end{definition}

Let us describe primitive extensions of multiplicity $k+1$ of type $\sL$. For $j=1,\ldots,k$, we have exact sequences
\begin{equation}\label{strSheaf}
0 \rightarrow \sL ^j \rightarrow \sO_{Y_j} \rightarrow \sO_{Y_{j-1}} \rightarrow 0.
\end{equation}
Moreover we also have the exact sequence
$$
0 \rightarrow \sL^{k}\to \frac{\sI_{Y}}{\sI_C \sI_Y} \rightarrow \frac{\sI_{C}}{\sI_C^2} \rightarrow \sL \rightarrow 0
$$
and in particular $\omega_{Y|C}\cong\omega_C \otimes \sL^{-k}$. Thus in order to effectively compute the canonical sheaf of $Y$, it is essential to understand the behavior of the restriction map $\Pic(Y) \to \Pic(C)$. The following paragraphs deal with these issues for primitive extensions of rational curves $C$ of type $\sO_C$ which will be related to instanton bundles.

\begin{lemma}\cite[Lemma 4.5]{AMMP}
\label{lPicMultiple}
Let $C\subset X$ be a rational curve and let $Y$ be a primitive extension of $C$ of type $\sO_C$. Then the restriction map $\Pic(Y) \to \Pic(C)$ is an isomorphism.
\end{lemma}

As a straightforward consequence of the above lemma, we infer that  the restriction $\omega_{Y|C}$ completely determines the canonical sheaf $\omega_Y$ in the case of primitive extensions of type $\sO_C$.
 
Now we will explicitly describe the normal bundle of a primitive extension $Y$ of type $\sO_C$ and multiplicity $k+1$ with support a rational smooth curve described in Proposition \ref{pCompleteIntersection}. In order to do so, following the notation introduced before, we recall the following two short exact sequences (we will simply use the notation $\sI_Y$ when considering the inclusion $Y\subset X$):  
\begin{gather}\label{eIdealMult}
0 \to  \frac{\sI_C^{k+1}}{\sI_{Y_1}\sI_{C}^k} \to \frac{\sI_Y}{\sI_C \sI_Y} \to \frac{\sI_{Y_1}}{\sI_C^2} \to 0,  \notag \\ 
\\ 
0 \to \frac{\sI_{Y_1}}{\sI_C^2} \to \frac{\sI_C}{\sI_C^2}\to \frac{\sI_C}{\sI_{Y_1}} \to 0. \notag
\end{gather}
Replacing the entry of the second short exact sequence according to the known isomorphisms, we get 
\begin{equation}\label{mult2}
0 \rightarrow \frac{\sI_{Y_1}}{\sI_C^2} \rightarrow \sO_{\pp^1}\oplus \sO_{\pp^1}(-2a) \rightarrow \sO_C \rightarrow 0.
\end{equation}

This implies that
$$
\frac{\sI_{Y_1}}{\sI_C^2} \simeq \sO_{\pp^1}(-2a)
$$
and 
$$
\frac{\sI_Y}{\sI_C \sI_Y} \simeq \frac{\sI_{Y_1}}{\sI_C^2} \oplus \frac{\sI_C^{k+1}}{\sI_{Y_1}\sI_{C}^k} \simeq \sO_{\pp^1}(-2a) \oplus \sO_{\pp^1}.
$$

This means that the restriction of the conormal bundle $\sN^{\vee}_{Y}$ to the curve $C$ is isomorphic to $\sO_{\pp^1}(-2a) \oplus \sO_{\pp^1}$, or, equivalently,
\begin{equation}
\label{eDetNormalMultiple}
\sN_{Y|X}  \otimes \sO_C\cong \sO_{\pp^1} \oplus \sO_{\pp^1}(2a).
\end{equation}
In order to determine $\sN_{Y|X}$, we can strengthen Lemma \ref{lPicMultiple}  by means of the following Proposition. 

\begin{proposition}\label{pSplitMultiple}
Let $C$ be a smooth curve satisfying the hypotheses of Proposition \ref{pCompleteIntersection} and let $Y$ be a  primitive extension of multiplicity $k+1$ and type $\sO_C$. Then any locally free sheaf on $Y$ splits.
\end{proposition}
\begin{proof}
See \cite[Proposition 4.6]{AMMP}.
\end{proof}
As a direct consequence of Lemma \ref{lPicMultiple}, Proposition \ref{pSplitMultiple} and Proposition \ref{pCompleteIntersection}, it follows that 
\begin{equation}
\label{eNormalMultiple}
\sN_{Y|X} \cong \sO_Y \oplus \sO_Y(2h_1).
\end{equation}
We now relax the condition in Definition \ref{dPrimitive} by considering quasi-primitive extensions of multiplicity $k+1\ge 2$. Let us recall the definition.

\begin{definition}\label{dTrivialThick} A multiple structure $Y$ on a smooth integral curve $C$ is called {\em quasi-primitive} if $Y$ is a Cohen-Macaulay curve such that $Y$ does not contain the first infinitesimal neighborhood $C^{(1)}$ of $C$. 
\end{definition}

In order to completely characterize locally complete intersection curves appearing as the zero locus of a section of the bundles we are interested in, we will make use of the following lemma.

\begin{lemma}\label{pInfinitesimal}
Let $C\subset X$ be as in Proposition \ref{pCompleteIntersection}. Then $C^{(1)}$ has the following $\sO_X$-resolution:
$$
\begin{array}{ccccccccc}
 & & & \sO_X(-2h_3)\\
     &  \sO_X(-h_1-ah_2-h_2) & & \oplus \\
 0 \arr & \oplus & \stackrel{M}\arr & \sO_X(-h_1-ah_2-h_3) & \arr & \sI_{C^{(1)}|X} & \arr 0  \\    
     &  \sO_X(-2h_1-2ah_2-h_3) & & \oplus  \\
 & & & \sO_X(-2h_1-2ah_2)
\end{array}
$$
where $M$ can be represented by the matrix
$$
\left(
\begin{array}{cc}
     \psi & 0  \\
     0 & \eta \\
     -\eta & \psi
\end{array}
\right)
$$
in which $\eta \in H^0(\sO_X(h_3))$ and $\psi \in H^0(\sO_X(h_1+ah_2))$ are the two generators of $I_C$. In particular, $\chi(\sO_{C^{(1)}})=3-2a$.
\end{lemma}
\begin{proof}
The statement follows directly from the fact that $I_{C^{(1)}|X} = \langle \eta^2, \psi^2, \eta\psi\rangle$ is a standard determinantal ideal defined by the maximal minors of the matrix representing $M$.
\end{proof}

To conclude this section, let us compute the genus of  a multiplicity $k+1$, quasi-primitive extension $Y$ of a rational curve $C$ as in Proposition \ref{pCompleteIntersection}. Consider the filtration \eqref{eFiltration} and write $\sI_j$ in place of $\sI_{Y_j}$. Thanks to \cite[Section 3]{BF},  $\sL_j:=\sI_{j-1}/\sI_{j}$ are line bundles on $C$. Furthermore, the maps
\[
\sL^{\otimes j} \to \sL_j
\]
are generically surjective, thus $\sL_j \cong \sL^{\otimes j} \otimes \sO_C(B_j)$ for some effective divisors $B_j$. Moreover, we have the short exact sequence
\begin{equation}\label{strSheaf2}
0 \rightarrow \sL_j \rightarrow \sO_{Y_j} \rightarrow \sO_{Y_{j-1}} \rightarrow 0,
\end{equation}
which yields
\begin{equation}\label{eChiMultiple}
\chi(\sO_Y) = \chi (\sO_C) + \sum_{j=1}^{k}{\chi(\sL_j)}.
\end{equation}
Since $\sL\cong \sO_{\pp^1}(\alpha)$ and $\sO_C(B_j)\cong \sO_{\pp^1}(b_i)$ for some $\alpha\geq 0$ and $b_i \ge0$, the equation \eqref{eChiMultiple} becomes
\begin{equation}\label{eGenusMultiple}
p_a(Y) = -\sum_{j=1}^{k}{(\alpha+1+b_i)}.
\end{equation}
Notice that a multiplicity $k+1$, quasi-primitive extension $Y$ is primitive of type $\sO_{\pp^1}(\alpha)$ if and only if $b_i=0$ for all $i$, and in this case $p_a(Y) = -k- \frac{k(k+1)}{2}\alpha$.

\section{General properties of vector bundles and instantons}
In this section we collect some general facts and properties about rank two vector bundles on $X$.

Suppose $\sG$ is a rank two vector bundle with Chern classes $c_1=a_1h_1+a_2h_2+a_3h_3$ and $c_2=k_1e_1+ k_2e_2+k_3e_3$. If we denote by $D$ the divisor $D=d_1h_1+d_2h_2+d_3h_3$  we have:

\begin{align}
c_1(\sG (D))=&\sum_{i=1}^3(a_i+2d_i)h_i \label{chernppp}\\
c_2(\sG(D))=&\sum_{i=1}^3{(k_i+a_jd_w+a_wd_j+2d_jd_w)e_i} \notag
\end{align}

for $i\neq j \neq w \neq i$ and $(d_1, d_2, d_3)\in \mathbb Z^{\oplus 3}$.

We will often make use of the Riemann-Roch formula in order to understand the cohomological behavior of the bundles we are interested in. On the threefold $X$ it takes the form: 
\begin{equation}\label{RR}
\chi (\sG)=(a_1+1)(a_2+1)(a_3+1)+1
-\frac{1}{2}(a_1k_1+a_2k_2+a_3k_3 +2(k_1+k_2+k_3))
\end{equation}

We fix the notion of $\mu$-semistability, which will play a central role in the study of instanton bundles. Let $\sF$ be a torsion free sheaf on $X$. We denote by $\mu_h(\sF)$ (or just $\mu(\sF)$ when no confusion arises) the {\sl slope} of $\sF$ with respect to $h$, i.e. the rational number $\mu(F):=c_1(\sF)h^2/\rank(\sF)$. We say that a vector bundle $\sE$ is {\sl $\mu$-stable} (resp. {\sl $\mu$-semistable}) with respect to $h$ if  $\mu( \sM) < \mu(\sE)$ (resp. $\mu(\sM) \le \mu(\sE)$) for each subsheaf $ M$ with $0<\rank(\sM)<\rank(\sE)$.

Let us recall the definition of instanton bundle (cf. \cite[Definition 2.1]{AnMa}) and $D$-'t Hooft bundle (cf. \cite[Definition 5.1]{AMMP})

\begin{definition}\label{dInstanton}
A $\mu$-semistable vector bundle $\sE$ on $\ppp$ is called an instanton bundle of charge $k$ if and only if $c_1(\sE)=0$,
\begin{equation*}H^0(\sE)=H^1(\sE(-h))=0
\end{equation*}
and $c_2(\sE)=k_1e_1+k_2e_2+k_3e_3$ with $k_1+k_2+k_3=k$.

Furthermore, given an effective divisor $D$ on $X$, an instanton bundle $\sE$ is a $D$\textit{-'t Hooft bundle} if and only if $h^0(\sE(D)) \neq 0$.
\end{definition}

Instanton bundles are strictly related to a another class of very interesting vector bundles, namely the Ulrich ones. Indeed one may think of instantons as an approximation of Ulrich bundles (cf. \cite{AnCa}). Let us briefly recall the definition of Ulrich bundles. We refer the interested reader to \cite{E-S-W} and \cite{CMRP} for more details on Ulrich bundles. 

\begin{definition}\label{dUlrich}
Let $Z\subset \pp^N$ be a smooth projective variety of dimension $n$ which is naturally endowed with the very ample line bundle $\sO_Z(h)=\sO_Z \otimes \sO_{\pp^N}(1)$. We say that a vector bundle $\sE$ is Ulrich if and only if
\[
H^i(\sE(-th))=0 \quad \text{for $i\ge 0$ and $1\le t \le n$}.
\]
\end{definition}
 
To describe this connection, notice that as in the classical case of the complex projective space $\pp^3$, every instanton bundle on $X$ can be represented as the cohomology of a monad.
\begin{theorem}\cite[Theorem 1.1 (i)]{AnMa}\label{tMonad}
Let $\sE$ be a charge $k$ instanton bundle on $X$ with $c_2(\sE)=k_1e_1+k_2e_2+k_3e_3$, then $\sE$ is the cohomology of a monad of the form

\begin{equation*}
0\rightarrow
\begin{matrix}
\sO_X^{k_3}(-h_1-h_2) \\
\oplus \\
\sO_X^{k_2}(-h_1-h_3) \\
\oplus \\
\sO_X^{k_1}(-h_2-h_3)
\end{matrix}
\rightarrow
\begin{matrix}
\sO_X^{k_2+k_3}(-h_1) \\
\oplus \\
\sO_X^{k_1+k_3}(-h_2) \\
\oplus\\
\sO_X^{k_1+k_2}(-h_3)
\end{matrix}
\rightarrow
\sO_X^{k-2}
\rightarrow
0.
\end{equation*}

\medskip
Conversely any $\mu$-semistable bundle defined as the cohomology of such  a monad is a  charge $k$ instanton bundle.
\end{theorem}

From Theorem \ref{tMonad} and Definition \ref{dInstanton} it follows that minimal charge instanton bundles (i.e. $k=2$) are actually Ulrich bundles (up to a twist). Indeed the cohomology group $H^1(\sE)$ has dimension $k-2$. Thanks to \cite[Proposition 2.1]{AnCa}, the vanishing of this cohomology group implies the vanishing of $H^1(\sE(th))$ for all $t$ and by Serre duality the same holds for $H^2(\sE(th))$. Since in this case we have that $\sE$, $\sE(-h)$ and $\sE(-2h)$ are acyclic, we obtain that $\sE(h)$ is a rank two Ulrich bundle.

Since instanton bundles are $\mu$-semistable by definition, we will often make use of Hoppe's criterion for semistable vector bundles over polycyclic varieties, i.e. varieties $X$ such that $\Pic(X)=\mathbb{Z}^l$.
\begin{proposition}\cite[Theorem 3]{JMPS}\label{pHoppe}
Let $\sE$ be a rank two holomorphic vector bundle over a polycyclic variety $X$ and let $L$ be a polarization on $X$. $\sE$ is $\mu$-(semi)stable if and only if
$$
H^0(X,\sE\otimes \sO_X(B))=0
$$
for all $B \in \Pic(X)$ such that $\delta_L(B) \underset{(<)}{\leq} -\mu_L(E)$, where $\delta_L(B)=\deg_L(\sO_X(B))$.
\end{proposition}
Thanks to this criterion one is able to classify strictly $\mu$-semistable instantons. 
\begin{proposition}\cite[Proposition 2.12]{AnMa}\label{pStrictlySemistable}
Let $\sE$ be an instanton bundle of charge $k$. If $\sE$ is not $\mu$-stable then $k=2\gamma^2$ for some $\gamma\in \mathbb{Z}$, $\gamma \neq 0$. Moreover $c_2(\sE)=2\gamma^2e_i$, $i=1,2,3$ and $\sE$ can be constructed as an extension
\begin{equation}\label{ext1}
0\rightarrow \sO_X(-\gamma h_i+\gamma h_j) \rightarrow \sE \rightarrow \sO_X(\gamma h_i-\gamma h_j)\rightarrow 0
\end{equation}
with $i \neq j$.
\end{proposition}

The previous proposition shows an actual method to construct instanton bundles. Nevertheless instanton obtained in this way are not $\mu$-stable, therefore it is necessary to resort to a different strategy for the construction of stable bundles.

Let $\sF$ be a rank two vector bundle on $X$ and let $s\in H^0\big(\sF\big)$. In general its zero--locus
$(s)_0\subseteq X$ is either empty or its codimension is at most
$2$. We can always write $(s)_0=Y\cup Z$
where $Z$ has codimension $2$ (or it is empty) and $Y$ has pure codimension
$1$ (or it is empty). In particular $\sF(-Y)$ has a section vanishing
on $Z$, thus we can consider its Koszul complex 
\begin{equation}
  \label{seqSerre}
  0\longrightarrow \sO_X(Y)\longrightarrow \sF\longrightarrow \sI_{Z\vert X}(-Y)\otimes\det(\sF)\longrightarrow 0.
\end{equation}
Sequence \ref{seqSerre} tensored by $\sO_Z$ yields $\sI_{Z\vert X}/\sI^2_{Z\vert X}\cong\sF^\vee(Y)\otimes\sO_Z$, whence the normal bundle $\sN_{Z\vert X}:=(\sI_{Z\vert X}/\sI^2_{Z\vert X})^\vee$ of $Z$ inside $X$ satisfies
\begin{equation}
\label{Normal}
\sN_{Z\vert X}\cong\sF(-Y)\otimes\sO_Z.
\end{equation}
If $Y=\emptyset$, then $Z$ is locally complete intersection inside $X$, because $\rk(\sF)=2$. In particular, it has no embedded components. The Serre correspondence reverts the above construction as follows.

\begin{theorem}
  \label{tSerre}
Let $X$ be an $n$--fold with $n\ge2$ and $Z\subseteq X$ a local complete intersection subscheme of codimension $2$ and let $\sL$ be a line bundle on $X$ such that $h^2\big(\mathcal L^\vee\big)=0$.

If $\Lambda^2(\sN_{Z\vert X})\otimes\mathcal L^\vee$ has a generating global section section, then there exists a rank two vector bundle $\sF$ on $X$ satisfying the following properties.
  \begin{enumerate}
  \item $\det(\sF)\cong\mathcal L$.
  \item $\sF$ has a section $s$ such that $Z$ coincides with the zero locus $(s)_0$ of $s$.
  \end{enumerate}
  Moreover, if $H^1\big({\mathcal L}^\vee\big)= 0$, the above two conditions  determine $\sF$ up to isomorphism.
\end{theorem}
\begin{proof}
See \cite{Ar}.
\end{proof}

\section{$h_i$-'t Hooft bundles}

Let $\sE$ be a $h_i$-'t Hooft bundle. The aim of this section is to characterize curves arising as the $0$-locus of a section of $\sE(h_i)$. Let us start by considering the short exact sequence
\begin{equation}\label{eh1tHooft}
0\to \sO_X\to \sE(h_1) \to \sI_{Y_1|X}(2h_1)\to 0.
\end{equation}
First of all, notice that the class of $Y_1$ is exactly $c_2(\sE)$ because $c_2(\sE)=c_2(\sE(h_1))$ by \eqref{chernppp}. Since $\sE$ is an instanton bundle, considering the long exact sequence induced in cohomology by Sequence \eqref{eh1tHooft}, we get $h^0(\sO_{Y_1}(-h_2-h_3))=h^1(\sE(-h))=0$. Moreover restricting \eqref{eh1tHooft} to $\sO_{Y_1}$ we obtain $\det(\sN_{Y_1|X})\cong \sO_{Y_1}(2h_1)$, thus adjunction formula yields $\omega_{Y_1}\cong\sO_{Y_1}(-2h_2-2h_3)$. Finally, the genus formula gives us $p_a(Y_1)=1-k_2-k_3$. Thus, our goal is to characterize locally complete intersection curves satisfying the above conditions. Let us start with the following lemma.

\begin{lemma}\label{lCanonicalMultiple}
 Let $C$  be the connected union of a smooth rational curve $Z$ representing $ae_1+e_i^2$ with $i \in \{2,3\}$ and a line $L$ representing $e_1$.
If $Y \subset X$ is a multiple structure supported on the curve $C$, then $\omega_Y \not\simeq \sO_Y(-2h_2-2h_3)$.
\end{lemma}
\begin{proof}
The curve $C$ satisfies $h^0(\sO_C(-h_2-h_3))=h^1(\sO_C(-h_2-h_3))=0$, however $\omega_C \not\cong \sO_C(-2h_2-2h_3)$. Now we want to show that 
for any multiple structure $Y$ on such a curve $C$, we also have $\omega_Y\not\cong\sO_Y(-2h_2-2h_3)$. The proof follows the lines of \cite[Lemma 5.5]{AMMP}, we recall the structure of the proof for the reader's convenience. Consider $\sF$ and $\sG$ the sheaves associated to $C$ and $Y$ through the Serre correspondence. Since $C\subset Y$ we get a short exact sequence
\begin{equation}\label{FtoG}
0 \to \sG \to \sF \to \sI_{C|Y}(2h_1)\to 0.
\end{equation}
Suppose by contradiction that $\omega_Y \cong \sO_Y(-2h_2-2h_3)$. The Serre correspondence yields that $\sG$ is a vector bundle. In particular, $\sE xt^1(\sG, \sO_X) = 0$. Applying $\sH om(-,\sO_F)$ to Sequence (\ref{FtoG}), we have an inclusion
$$
0\neq\sE xt^1(\sF, \sO_X) \hookrightarrow \sE xt^1(\sG, \sO_F)=0,
$$
since $\sE xt^i(\sI_{C|Y}(h_i),\sO_F)=0$ for $i=0,1$ by \cite[III 7.3]{Har}, leading to contradiction. Therefore the canonical sheaf of $Y$ cannot have the considered form, proving our result.
\end{proof}
Notice that the curves described in the previous lemma give rise to torsion free sheaves satisfying all the instantonic conditions but locally freeness. We are ready to state the main result of this section.

\begin{theorem}\label{tCurves}
Let $Y\subset X$ be a locally complete intersection curve. The following statements are equivalent:
\begin{enumerate}
\item [(i)]$\omega_Y\cong\sO_Y(-2h_2-2h_3)$ and $h^0(\sO_Y(-h_2-h_3))=0$.
\item [(ii)]$Y$ is the disjoint union of primitive extensions of type $\sO_C$ on smooth rational curves $C$ representing the class $ae_1+e_i$ with $i\in\{2,3\}$ and $a\ge 0$.
\end{enumerate}
\end{theorem}
\begin{proof}
$(i)\Rightarrow(ii):$ 

Since conditions in $(i)$ hold for any connected component of $Y$, we can suppose $Y$ to be a connected curve. Let $C:=Y_{red}$ and consider the natural injective map
\begin{equation}\label{eInjCanonical}
0 \to \omega_C\to \omega_Y.
\end{equation}
Since $h^0(\omega_Y(h_2+h_3))=0$ by hypothesis we get $h^0(\omega_C(h_2+h_3))=0$. 
Suppose now that $C$ is reducible and write $C=C_1 \cup C_2$, with $C_1$ irreducible. From the short exact sequence
\begin{equation}\label{eMV}
0 \to \omega_{C_1} \oplus \omega_{C_2} \to \omega_{C} \to \omega_{C_1 \cap C_2} \to 0
\end{equation}
we also have $h^0(\omega_{C_1}(h_2+h_3))=0$, thus $h^0(\omega_{C_1})=0$. As $C_1$ is integral, we have
$$
p_a(C_1)=1-\chi(\sO_{C_1})=1-h^0(\sO_{C_1})+h^0(\omega_{C_1})=0,
$$
i.e., $C_1$ is a smooth rational curve. In order to compute its class, suppose that $C_1$ has class $a_1e_1+a_2e_2+a_3e_3$ in $A^2(X)$. By Riemann-Roch we have
$$
0\geq \chi(\omega_{C_1}(h_2+h_3))=2p_a(C_1)-2+\deg_{\sO_X(h_2+h_3)}(C_1)+1=-1+a_2+a_3.
$$

Therefore, any irreducible component of $C$ has $0\leq a_2+a_3\leq 1$. Moreover, tensoring Sequence \eqref{eMV} by $\sO_X(h_2+h_3)$ gives $h^0(\omega_{C_1 \cap C_2})\leq h^1(\omega_{C_2}(h_2+h_3))$ when $a_2+a_3=1$ and $h^0(\omega_{C_1 \cap C_2})\leq h^1(\omega_{C_2}(h_2+h_3))+1$ when $a_2+a_3=0$.

We are going to prove that $C$ consists of just one irreducible component with $a_2+a_3=1$. First of all, if there is no such component, $C$ would be the union of some irreducible components $C_i$, $i=1,\ldots,r$, representing $e_1$ in $A^2(X)$. Since they are disjoint pairwise, we have $r=1$ and therefore $Y$ would also represent $ae_1$ for some $a\geq 1$. However, this is impossible, since in this case $\sO_Y(-h_2-h_3)\cong\sO_Y$ and, in particular, we would get that $h^0(\sO_Y(-h_2-h_3))\neq 0$.

On the other hand, if $C$ contains two irreducible components $C_1$ and $C_2$ with classes $a_1e_1+e_i$ and $a_2e_1+e_j$ with $i,j\in \{2,3\}$, again by means of the short exact sequence \eqref{eMV}, we see that $h^0(\omega_{C_1\cap C_2})=0$, that is $C_1$ and $C_2$ are disjoint. Since they are components of the connected curve $C$, there should exist a third irreducible component $Z$ of class $ae_1$ connecting them, and in particular intersecting their union in at least two points. But again, the exact sequence 
$$
0 \to \omega_{C_1\cup C_2} \oplus \omega_{Z} \to \omega_{C_1\cup C_2\cup Z } \to \omega_{(C_1\cup C_2) \cap Z} \to 0
$$
implies that $h^0(\omega_{(C_1\cup C_2) \cap Z})\leq 1$, a contradiction.

To complete the argument and exclude the case where $C$ is the reducible union of two curves, we can apply directly Lemma \ref{lCanonicalMultiple} to conclude that $C$ is an irreducible curve with class $ae_1+e_i$ with $i \in \{2,3\}$ and $a\geq 0$.

Let us show now that a curve $Y$ satisfying $(i)$ and with the reduced structure described above is quasi-primitive. By Definition \ref{dTrivialThick}, it suffices to show that $Y$ does not contain the first infinitesimal neighborhood $C^{(1)}$ of $C$. Let us take the exact sequence
$$
0\to \sN^{\vee}_{C\mid X}\to\sO_{C^{(1)}}\to\sO_C\arr 0.
$$
By Proposition \ref{pCompleteIntersection}, we have $\sN^{\vee}_{C\mid F}\cong\sO_{\pp^1}\oplus\sO_{\pp^1}(-2a)$. In particular, $h^1(\sO_{C^{(1)}}(-h_2-h_3))\neq 0$. Therefore, if $C^{(1)}\subset Y$, we would have a surjection
$$
H^1(\sO_Y(-h_2-h_3))\arr H^1(\sO_{C^{(1)}}(-h_2-h_3))\neq \{0\}
$$
contradicting $h^1(\sO_Y(-h_2-h_3))=h^0(\sO_Y(-h_2-h_3))=0$. In order to see that $Y$ is actually primitive, consider the filtration
\begin{equation}\label{filt}
    C=Y_0 \subset Y_1 \subset \dots \subset Y_k=Y
\end{equation}
and observe that from the short exact sequence
$$
0\to\frac{\sI_{Y_1}}{\sI_{C^2}}\to \mathcal{N}^{\vee}_{C\mid X}\cong\sO_{\pp^1}\oplus \sO_{\pp^1}(-2a)\to \sO_{\pp^1}(\alpha)\to 0
$$
we obtain $\alpha\geq 0$.  Computing
$$
0=\chi(\sO_Y(-h_2-h_3))=-\mathrm{mult}_C(Y)+1-p_a(Y).
$$
and putting this information in  Formula \eqref{eGenusMultiple}, we get $\alpha=b_t=0$ for all $t$, hence $Y$ is a primitive extension of type $\sO_C$.

$(ii)\Rightarrow(i)$:

Let $Y$ be a primitive extension of type $\sO_C$ as in $(ii)$. Thanks to \eqref{eNormalMultiple}, adjunction formula yields $\omega_Y\cong\sO_Y(-2h_2-2h_3)$. 
As a last step, using recursively the short exact sequences
$$
0 \longrightarrow \sO_C^{\oplus r_k} \longrightarrow \sO_{Y_k} \longrightarrow \sO_{Y_{k-1}} \longrightarrow 0
$$
we conclude by induction that $h^0(\sO_Y(-h_2-h_3))=0$.
\end{proof}

Our next goal is to describe the open locus, inside the Hilbert scheme of curves representing the class $k_1e_1+k_2e_2+k_3e_3$ of degree $k=k_1+k_2+k_3$ and arithmetic genus $p_a(Y)=1-k_2-k_3$, of locally complete intersection curves satisfying the hypotheses of Theorem \ref{tCurves}.

\begin{proposition}\label{pHilbertComponents}
Let $H \subset \mathcal{H}:=\mathrm{Hilb}^{kt+k_2+k_3}(X)$ be the Hilbert scheme of curves $Y$ representing $k_1e_1+k_2e_2+k_3e_3$ of degree $k=k_1+k_2+k_3$ and arithmetic genus $p_a(Y)=1-k_2-k_3$ satisfying $\omega_Y\cong \sO_Y(-2h_2-2h_3)$ and $h^0(\sO_Y(-h_2-h_3))=0$. Then $H$ is smooth  of dimension $2k$. The number of its irreducible components is given by 
\begin{equation}\label{eNumberComponents}
\Delta(k_1e_1+k_2e_2+k_3e_3):=
\mathlarger{\mathlarger{\sum}}_{{\substack{ 0\le \alpha \le k_2 \\ 0 \le \beta \le k_3 \\ 1 \le \alpha + \beta \le k_1}}}{\binom{k_1-1}{\alpha+\beta-1}},
\end{equation}
if $k_1\ge 1$ and it is irreducible if $k_1=0$.
\end{proposition}
\begin{proof}
Let us deal with $H_{red}\subset H$ which is the subset of $H$ consisting of points representing reduced curves, since the numerology of both cases is completely analogous. Observe that an element $Y$ of $H_{red}$ representing $k_1e_1+k_2e_2+k_3e_3$ is the disjoint union of $k_2+k_3$ rational curves representing the class $a_ie_1+e_2$ or $b_je_1+e_3$ with $i=1,\dots,k_2$ and $j=i,\dots,k_3$. Thanks to Proposition \ref{pCompleteIntersection}, we get 
\[
h^0(\sN_{Y|X})=2k \quad \text{and} \quad h^1(\sN_{Y|X})=0.
\]
Notice that the result holds also in the non-reduced case, thanks to Equation \eqref{eNormalMultiple}. Thus $H$ is smooth and each irreducible component of $H$ has dimension $2k$. In order to show that $H$ consists of $\Delta(k_1e_1+k_2e_2+k_3e_3)$ irreducible components, suppose that $Y_1$ and $Y_2$ are distinct points of $H$. Then $h^0(\sO_{Y_1}(-h_1-h_2))=h^0(\sO_{Y_2}(-h_1-h_2))$ and $h^0(\sO_{Y_1}(-h_1-h_3))=h^0(\sO_{Y_2}(-h_1-h_3))$ if and only if $Y_1$ and $Y_2$ represent the same configuration type of rational curves (not necessarily the same curve). Equivalently, curves representing different configurations do not satisfy the previous equalities. The semi-continuity of cohomology yields that $Y_1$ and $Y_2$ must lie in different irreducible components, since $h^0(\sN_{Y_1|X})=h^0(\sN_{Y_2|X})=2k$. The number of components follows by counting the possible ways of partitioning $k_1e_1+k_2e_2+k_3e_3$ into the sum of $a_ie_1+e_2$ and $b_je_1+e_3$ with $a_i,b_j\ge 0$.
\end{proof}
As a direct consequence, by imposing that \eqref{eNumberComponents} is equal to $1$, we get the following corollary.
\begin{corollary}\label{cIrreducible}
$H$ is irreducible if and only if $k_1=0$ or $k_i=k-1$ for some $i\in\{1,2,3\}$.
\end{corollary}

After describing the curves arising as the zero locus of a section of $\sE(h_i)$, we turn our attention back to the $h_i$-'t Hooft bundle. Once we require that $h^0(\sE(h_i))>0$ it is quite natural to ask whether this dimension is bounded by some function depending on the invariants of $\sE$. In the next proposition we find a bound for the space of global sections of $\sE(h_i)$.
\begin{proposition}\label{pBound}
Let $\sE$ be an instanton bundle, then $h^0(\sE(h_i))\le 2$.
\end{proposition}
\begin{proof}
Suppose $\sE$ is a $h_1$-'t Hooft bundle, so that $H^0(\sE(h_1))\neq 0$. From Sequence \eqref{eh1tHooft}, we deduce $h^0(\sE(h_1))=1+h^0(\sI_{Y_1|X}(2h_1))$. Thus it is enough to count how many surfaces in the linear system $|2h_1|$ contain $Y_1$. Notice that any element $S$ of $|2h_1|$ is either the disjoint union of two quadrics or a double structure supported on a quadric. Notice that $Y_1$ is never contained in a quadric, for otherwise $H^0(\sE)\cong H^0(\sI_{Y_1|X}(h_1))\neq 0$ and being $\sE$ an instanton bundle, it has no global sections. Moreover any curve representing the class $ae_1+e_i$ with $i\neq 1$ and $a \ge 0 $ is never contained in a quadric in the linear system $|h_1|$. A curve with class $ae_2$ is contained in a unique $h_1$-quadric (if any) and the same holds for $be_3$, thus $h^0(\sI_{Y_1|X}(2h_1))\le 1$, from which the statement follows. The multiple case is completely analogous, since the only rational curves contained in a double quadric are the ones contained in the support or their multiple thickening.
\end{proof}
From the previous proposition one can deduce a couple of direct, yet interesting, corollaries.
\begin{corollary}\label{cQuarticSectional}
Let $\sE$ be an instanton bundle. Then $h^0(\sE(h_1))= 2$ if and only if given $s \in H^0(\sE(h_1))$, $Y_1=(s)_0$ is contained in a unique quartic surface $S\in|2h_1|$.
\end{corollary}
\begin{proof}
From Proposition \ref{pBound} and its proof it follows directly that $h^0(\sE(h_1))=2$ if and only if $h^0(\sI_{Y_1|X}(2h_1))=1$, i.e. $Y_1$ is contained in a surface in the linear system $|2h_1|$.\end{proof}
\begin{corollary}
Let $Y$ be a curve arising as the zero locus of a section of $\sE(h_1)$, then $Yh_1\neq0$ implies $h^0(\sE(h_1))=1$.
\end{corollary}
\begin{proof}
 The hypothesis $Yh_1\neq 0$ imply that $Y$ is not contained in any element of the linear system $|h_1|$, thus the statement directly follows from Sequence \eqref{eh1tHooft}.
\end{proof}

We end this section by showing that one can actually reverse the relation between $h_i$-'t Hooft instanton bundle and locally complete intersection curves $Y$ satisfying $\omega_Y\cong \sO_Y(-2h_j-2h_w)$ and $h^0(\sO_Y(-h_j-h_w))=0$ with $ i\neq j \neq w \neq i$.

\begin{proposition}\label{pSerre}
    Let $Y$ be a locally complete intersection curve as in Theorem \ref{tCurves} such that $Y$ is not a line. Then every non-trivial element of $\Ext^1(\sI_{Y|X}(2h_1),\sO_X)$ is a $\mu$-stable $h_1$-'t Hooft bundle.
\end{proposition}
\begin{proof}
Let $\sE$ a non-trivial element $\Ext^1(\sI_{Y|X}(2h_1),\sO_X)$. Then $\sE$ is a vector bundle thanks to Theorem \ref{tSerre}. Let us start by proving that $h^0(\sE)=h^1(\sE(-h))=0$. Notice that $\sE$ fits into \eqref{eh1tHooft}, thus tensoring it by $\sO_X(-h_1)$ and $\sO_X(-2h_1-h_2-h_3)$ and considering the long exact sequence induced in cohomology, one gets the desired vanshings. It remains to prove that $\sE$ is $\mu$-stable and we will make use of Proposition \ref{pHoppe}. Let $D=d_1h_1+d_2h_2+d_3h_3$ a divisor of negative degree. We show that $\sE(D)$ has no global sections. Indeed the cohomology of Sequence \eqref{eh1tHooft} tensored by $\sO_X(D-h_1)$ returns
\[
h^0(\sE(D))\le h^0(\sI_{Y|X}(D+h_1))+h^0(\sO_X(D-h_1)).
\]
Notice that since at least one $d_i$ is negative, trivially $h^0(\sO_X(D-h_1))=0$. Using the short exact sequence 
\[
0 \to \sI_{Y|X} \to \sO_X \to \sO_Y \to 0
\]
it is straightforward to see that $h^0(\sI_{Y|X}(D+h_1))=0$. Notice that if $Y$ is a line, it will represent the class $e_2$ or $e_3$. In both cases $h^0(\sI_{Y|X}(h_1))=1$ which implies $h^0(\sE)\neq 0$, thus the proof is complete.
\end{proof}
Thanks to the previous results, we are able to fully characterize $\mu$-stable $h_i$-'t Hooft bundles by means of the zero locus of a section of $\sE(h_i)$. Thus it is interesting to ask whether strictly $\mu$-semistable instanton bundles can be $h_i$-'t Hooft.
\begin{proposition}\label{pSStHooft}
Let $\sE$ be a strictly $\mu$-semistable instanton bundle with $c_2(\sE)=2\gamma^2 e_i$. Then $\sE$ is a $h_j$-'t Hooft bundle if and only if $j \neq i$ and $\gamma=1$.
\end{proposition}
\begin{proof}
The result follows trivially by considering the long exact sequence induced in cohomology by Sequence \eqref{ext1} twisted by $\sO_X(h_j)$ for $j=1,2,3$.
\end{proof}

\section{Special 't Hooft bundles}

In this section we introduce the notion of special bundles.  
Given an effective divisor $D$, one may consider various subclasses of
$D$-'t Hooft bundles. On the one hand, one can fix another effective divisor
$B$ and require that an instanton bundle $\sE$ satisfy the 't Hooft
condition with respect to both $D$ and $B$. On the other hand, within the
class of $D$-'t Hooft bundles, one may single out those for which
$H^0(\sE(D))$ has maximal possible dimension. Motivated by this dichotomy, we introduce the following definitions (cf. \cite[Definition 7.1]{AMMP}).

\begin{definition}\label{dSpecial1}
We say that an instanton bundle is $(D_i,D_j)$-special if and only if it is both 't Hooft with respect to $D_i$ and $D_j$.
\end{definition}
\begin{definition}\label{dSpecial2}
We say that an instanton bundle is $(D,D)$-special or $D$-sectional special if and only if $h^0(\sE(D))>1$.
\end{definition}
Notice that in the case $D=h_i$, being $h_i$-sectional special is equivalent to require that $h^0(\sE(h_i))=2$ thanks to Proposition \ref{pBound}.

We start dealing with the $(D_i,D_j)$-special case, where $D_w=h_w$. For simplicity of notation let us discuss the $(h_1,h_2)$ case. Suppose $\sE$ is a $h_1$-'t Hooft bundle and tensor sequence \eqref{eh1tHooft} by $\sO_X(-h_1+h_2)$, then we get
\begin{equation}\label{eSpecial12}
0\to \sO_X(-h_1+h_2) \to \sE(h_2) \to \sI_{Y_1}(h_1+h_2) \to 0.
\end{equation}
From this, the next Proposition follows directly.
\begin{proposition}\label{pQuartic}
An instanton bundle $\sE$ is $(h_1,h_2)$-special if and only if is $h_1$-'t Hooft (resp. $h_2$-'t Hooft) and $Y_1$ (resp. $Y_2$) is contained in a quartic surface $S$ in $|h_1+h_2|$, where $Y_1=(s_1)_0$ and $s_1 \in H^0(\sE(h_1))$ (resp. $Y_2=(s_2)_0$ and $s_2 \in H^0(\sE(h_2))$).
\end{proposition}

This strong geometric constraint allows for a complete classification of curves arising as vanishing locus of a section of $(h_1,h_2)$-special bundles.

\begin{theorem}\label{tSpecial}
Let $\sE$ be a $(h_1,h_2)$-special 't Hooft bundle, and let $Y_1$ be the zero locus of a section of $\sE(h_1)$. Then one of the following holds:
\begin{itemize}
    \item[(i)] $Y_1$ represents the class $ke_3$ with $k>0$ and it is the disjoint union of (possibly multiple) lines;
    \item [(ii)] $Y_1$ represents the class $k(e_1+e_2)$ and it is the disjoint union of (possibly multiple) conics;
    \item [(iii)] $Y_1$ represents the class $k_2e_2+k_3e_3$ with $k_2,k_3> 0$ and it is the disjoint union of (possibly multiple) lines;
    \item [(iv)] $Y_1$ represents the class $k_1e_1+k_3e_3$ with $k_i>0$ and it is the disjoint union of a rational curve with class $k_1e_1+e_3$ and (possibly multiple) lines;
    \item [(v)] $Y_1$ represents the class $k_1e_1+k_2e_2+e_3$ with $k_i>0$ and it is the disjoint union of a rational curve with class $k_1e_1+e_3$ and (possibly multiple) lines.
\end{itemize}
Moreover in case $(iii)$, $(iv)$ and $(v)$ the associated quartic is the reducible union of two quadrics.
\end{theorem}
\begin{proof}
The proof boils down to understand which of the curves appearing in Theorem \ref{tCurves} can be contained in a surface $S$ in the linear system $|h_1+h_2|$. Notice that such $S$ can either be irreducible (and smooth, thanks to Lemma \ref{lQuartics}) or the reducible union of two quadrics $Q_1$ and $Q_2$. Let us start by dealing with the irreducible case: $S$ has two rulings, one made by conics representing the class $e_1+e_2$ in $X$ and the class $m$ in $S$ and the other made by lines of type $e_3$ in $X$ and $\ell$ in  $S$. Notice that any irreducible rational curve in $S$ represents the class $e_1+e_2+\alpha e_3$ or $\beta(e_1+e_2)+e_3$ for some $\alpha$ and $\beta$ in $\mathbb{Z}_{\ge0}$, thus the only curves described in Theorem \ref{tCurves} contained in $S$ are the ones with $\alpha=0$ or $\beta=0$. Notice that since they are the generators of the two rulings we can either take in $S$ an element in $|k\ell|$ or an element  in $|k m|$, so this proves $(i)$ and $(ii)$. Now we deal with the reducible case, so $S=Q_1\cup Q_2$ with $Q_i\in|h_i|$. Due to Lemma \ref{lQuadrics}, $Q_1$ has two rulings generated by lines $e_2$ and $e_3$, while the rulings of $Q_2$ are generated by $e_1$ and $e_3$. In particular, an irreducible rational curve in $Q_1$ has either class $\alpha e_2+e_3$ or $e_2+\beta e_3 $ for some $\alpha,\beta\ge0$, while rational curves in $Q_2$ represent the classes $\gamma e_1+e_3$ or $e_1+\delta e_3$ for some $\gamma,\delta \ge 0$. Because of this, $(iii)$, $(iv)$ and $(v)$ now follow from Theorem \ref{tCurves}.
\end{proof}
\begin{remark}
Theorem \ref{tSpecial} provides a numerical condition to check (non-)speciality of an instanton bundle. Indeed, suppose $\sE$ is a $h_1$-'t Hooft bundle and let $s\in H^0(\sE(h_1))$ be a global section. Then if $Y_1=(s)_0$ represents $k_1e_1+k_2e_2+k_3e_3$ in $A^2(X)$, with $k_i >1$ for all $i$, then $\sE$ is neither $(h_1,h_2)$-special nor $(h_1,h_3)$-special.
\end{remark}
It is natural to ask whether a $h_1$-'t Hooft bundle can admit sections when twisted by both $h_2$ and $h_3$, i.e.  is there an instanton bundle such that $h^0(\sE(h_i))>0$ for all $i$?

\begin{proposition}\label{pVerySpecial}
Let $\sE$ be a $h_1$-'t Hooft bundle of charge $k$. Then $\sE$ is $(h_1,h_2)$-special and $(h_1,h_3)$-special if and only if $k=2$ and $c_2(\sE)=e_i+e_j$ with $i \neq j$.
\end{proposition}
\begin{proof}
Let us start by supposing that $\sE$ is both $(h_1,h_2)$-special and $(h_1,h_3)$-special. Thanks to Proposition \ref{pQuartic} we need to show which are the cases appearing in Theorem \ref{tSpecial} which are contained in an element $S$ of $|h_1+h_3|$. We will do a case by case analysis. As per usual, we denote by $Y_1$ the zero locus of a section of $\sE(h_1)$. Notice that case $(i)$ cannot be contained in any $S\in|h_1+h_3|$, since if $S$ is irreducible it has two rulings generated by $e_2$ and $e_1+e_3$, while if it is reducible and contains $Y_1$ then the curve $Y_1$ would be contained in a quadric $h_1$, contradicting the semistability of $\sE$. A totally analogous argument yields that case $(v)$ cannot occur. Case $(ii)$ can only be admissible if $Y_1$ is a unique conic representing $e_1+e_2$ since it is contained in a irreducible quartic $S'\in |h_1+h_2|$ and in a reducible quartic $S=Q_1\cup Q_3 \in |h_1+h_3|$. Similarly, case $(iii)$ can only occur if $Y_1$ is the disjoint union of two lines $e_2$ and $e3$ it is also contained in a reducible quartic $S=Q_1\cup Q_3$. Finally case $(iv)$ can only be obtained when $Y_1$ is an irreducible conic representing $e_1+e_3$, since this conic is contained in a quadric $Q_2\in|h_2|$ (thus also in a reducible quartic in $|h_1+h_2|$) and in a irreducible quartic $S\in |h_1+h_3|$ since it generates one of its rulings.

For the inverse implication, Formula \eqref{RR} returns $\chi(\sE(h_i))>0$ for all $i$. The result then follows by noticing that $h^j(\sE(h_i))=0$ for $i=1,2,3$ and all $j>1$.
\end{proof}
Now we turn our attention to $D$-sectional special bundles, where $D=h_i$. Similarly to the $(h_1,h_2)$-special case, the curve $Y_1=(s)_0$ with $s \in H^0(\sE(h_1))$ is contained in a quartic surface (see Corollary \ref{cQuarticSectional}). Thanks to this observation, we are able to classify the curves associated with $h_i$-sectional special bundles.

\begin{theorem}\label{tSpecialSectional}
Let $\sE$ be a $h_1$-sectional special instanton bundle. Then $c_2(\sE)=k_2e_2+k_3e_3$ and the zero locus of a section of $\sE(h_1)$ is the disjoint union of lines.
\end{theorem}
\begin{proof}
From the proof of Proposition \ref{pBound} it follows that $\sE$ is $h_1$-sectional special if and only if the zero locus $Y_1$ of a section of $\sE(h_1)$ is contained in a quartic surface $S\in|2h_1|$. Again, we deal with the reduced case, the multiple case being analogous. In particular $S=Q_1\cup Q_2$ is the disjoint union of two quadrics. Since a curve in $S$ is represented in the Chow ring by $k_2e_2+k_3e_3$, it follows that $c_2(\sE)=k_2e_2+k_3e_3$. Then Theorem \ref{tCurves} yields that $(Y_1)_{red}$ is the disjoint union of lines having class $e_2$ or $e_3$, and lines representing different class should lie in different quadrics.
\end{proof}

Given the rich geometry and Chow ring structure of $X$, which give rise to these two notions of speciality, we investigate whether the concepts introduced in Definitions \ref{dSpecial1} and \ref{dSpecial2} can actually coincide.
\begin{proposition}\label{pDoubleSpecial}
Let $\sE$ be a $(h_1,h_2)$-special instanton of charge $k$ and let $Y_1=(s)_0$ with $s\in H^0(\sE(h_1))$. Then
\begin{itemize}
\item $\sE$ is both $h_1$-sectional special and $h_2$-sectional special if and only if $Y_1$ is the disjoint union of two lines in the family $e_3$.
\item $\sE$ is $h_1$-sectional special but not $h_2$-sectional special if and only if $Y_1$ represents the class $Y_1=ke_2+e_3$ with $k\ge1$ or $Y_1=ke_3$ with $k\ge 3$ and it is the disjoint union of (possibly multiple) lines with $k-1$ of them lying on a quadric $Q\in|h_1|$. 
\item $\sE$ is $h_2$-sectional special but not $h_1$-sectional special if and only if $Y_1$ is an irreducible rational curve representing the class $ke_1+e_3$ with $k\ge 1$ or $Y_1=ke_3$ with $k\ge 3$ and $Y_1$ is contained in a quadric $Q\in|h_2|$.
\end{itemize}
\end{proposition}
\begin{proof}
The proof follows the same lines as the one in Proposition \ref{pVerySpecial}. Suppose that the quartic $S\in |h_1+h_2|$ containing $Y_1$ is irreducible. Proposition \ref{pBound} yields that $\sE$ is $h_1$-sectional special if $Y_1$ is contained in the intersection between $S$ and the disjoint union of two quadrics in the linear system $|h_1|$. However, each quadric intersects $S$ along a line in the family $e_3$, so we get item $(i)$. Now suppose that $S$ is reducible. Case $(i)$ of Theorem \ref{tSpecial} yields that $Y_1$ is the disjoint union of (possibly multiple) lines representing $e_3$. If $\sE$ is $h_1$-sectional special, Corollary \ref{cQuarticSectional} gives us that $Y_1$ is contained in a disjoint union of two quadrics in $|h_1|$. In particular we have a single line on a quadric and the remaining ones lie on the other quadric. As soon as $Y_1$ represents $ke_3$ with $k\ge 3$ then $Y_1$ cannot be contained in the union of two quadrics in $|h_2|$ because quadrics in different linear systems intersect along a single line.

The case $(iii)$ of Theorem \ref{tSpecial} implies that $Y_1$ represents $k_2e_2+k_3e_3$ and its reduced structure is the disjoint union of lines. But the disjoint union of two quadrics in the linear system $|h_1|$ contains $Y_1$ only if $k_3=1$.

To finish the proof one should observe that if case $(iv)$ of Theorem \ref{tSpecial} occurs, then $\sE$ can only be $h_2$-sectional special and it is never $h_1$-sectional special. Indeed in this case $Y_1$ is an irreducible rational curve representing the class $ke_1+e_3$ for some $k\ge 1$. Such a curve is contained in a quadric $Q_2\in |h_2|$. In particular it is contained in a pencil of reducible surfaces in $|h_1+h_2|$ of affine dimension two. Twisting Sequence \eqref{eh1tHooft} by $\sO_X(-h_1+h_2)$ we get $h^0(\sE(h_2))=h^0(\sI_{Y_{1}}(h_1+h_2))=2$. To obtain the last case, we notice that the only other possible curve which can be contained in a quadric in $|h_2|$ is the disjoint union of $k$ (possibly multiple) lines in the family $e_3$.
\end{proof}
\begin{remark}\label{rDoubleSpecial}
In light of Proposition \ref{pDoubleSpecial} and Theorem \ref{tSpecialSectional} it is possible to find both $h_1$-sectional special and $h_2$-sectional special which are not $(h_1,h_2)$ special. For instance one can consider the $Y_1$ as the union of four lines representing $2e_2+2e_3$ in the Chow ring. We can choose the lines in the family $e_2$ in a quadric $Q_1 \in |h_1|$ and the ones in the family $e_3$ in a disjoint quadric $Q_1'$ in the same linear system. Then the bundle associated to $Y_1$ via Proposition \ref{pSerre} is not $(h_1,h_2)$-special since it is not contained in any quartic in the linear system $|h_1+h_2|$. Indeed since $2e_3$ is contained $Q_1'$, it cannot lie on a quadric in $Q_2\in|h_2|$, since $Q_1'\cap Q_2$ is just a line.
\end{remark}

We end this section by characterizing $h_1$-sectional special instantons  and $(h_1,h_2)$-special instantons via their dependency quartic.

\begin{proposition}\label{pDependencySectional}
Let $\sE$ be a $h_1$-sectional special instanton bundle. The dependence locus of two independent sections $s_1 \in H^0(\sE(h_1))$, $s_2\in H^0(\sE(h_1))$ is a reducible quartic surface $S \subset F$ in the linear system $|2h_1|$. Moreover $\sE(h_1)$ sits in the short exact sequence
\begin{equation}\label{eSectionalEvaluation}
0 \to \sO_X^{\oplus 2} \to \sE(h_1) \to \sO_{S}(-D_1)\to 0.
\end{equation}
where $D_1$ is the divisor associated to the zero locus $Y_1$ of $s_1$ in $\Pic(S)$. The surface $S$ is called the dependency quartic of $\sE$.
\end{proposition}
\begin{proof}
Let us denote by $Y_i$ the zero locus of the section $s_i$. Notice that he image of the non-zero global section $s_2\in H^0(\sE(h_1))$ in $H^0(\sI_{Y_1|X}(2h_1))$ corresponds to a surface $S$. This can be seen by gathering together the exact sequences given by the sections $s_1,s_2$ in the following commutative diagram: 
\begin{equation}\label{eGeneralEvaluation}
\xymatrix{
& 0 \ar[d] & 0 \ar[d]\\
& \sO_X \ar[r]^\cong \ar[d] & \sO_X \ar[d]^-{s_1}\\ 
0 \ar[r] & \sO_X^{\oplus 2}\ar[r]^-{(s_2,s_1)} \ar[d] & \sE(h_1) \ar[r] \ar[d] & \sI_{{Y_1}\mid S}(2h_1)  \ar[r]  \ar[d]_\cong& 0\\
0 \ar[r] & \sO_X \ar[r] \ar[d] & \sI_{Y_1|X}(2h_1) \ar[r] \ar[d] & \sI_{{Y_1}\mid S}(2h_1)  \ar[r] & 0.\\
& 0 & 0
}
\end{equation}
From it, we see that the locus where the two sections $s_1$ and $s_2$ are not independent is the surface $S$ whose defining equation induces the injective morphism from the lowest row of the previous diagram.
 Since $\sI_{Y_1\mid S}(2h_1)$
 has depth two at any point of $S$, we see that $\sI_{Y_1\mid S}(2h_1)$ is a Cohen-Macaulay sheaf. Therefore, it should be free at any regular point of $S$. It follows that $Y_1\cap Y_2$ is empty. In particular $\sI_{Y_1\mid S}(2h_1) $ is the line bundle $\sO_{S}(-D_1)$.
 \end{proof}
Arguing in the same way, we are able to describe the case of $(h_1,h_2)$-special bundles.
\begin{proposition}\label{pDependencySpecial}
Let $\sE$ be a $(h_1,h_2)$-special instanton bundle having $c_2(\sE)=k_1e_1+k_2e_2+k_3e_3$. The dependency locus of two independent sections $s_1 \in H^0(\sE(h_1))$, $s_2\in H^0(\sE(h_2))$ is a quartic surface $S \subset F$ in the linear system $|h_1+h_2|$. Moreover if $S$ is irreducible $\sE(h_1+h_2)$ sits in the short exact sequence
\begin{equation}\label{eSpecialEvaluation}
0 \to \sO_X(h_1) \oplus \sO_X(h_2) \to \sE(h_1+h_2) \to \sO_{S}((3-k_2-k_3)\ell-k_1m)) \to 0,
\end{equation}
where $S\cong (\pp^1 \times \pp^1, \sO_{\pp^1 \times \pp^1}(2\ell+m))$.
\end{proposition}
\begin{proof}
We omit the details of the proof which follows the same argument of Proposition \ref{pDependencySectional}. We just observe that Lemma \ref{lQuartics} yields $\sI_{Y_1|S}(2h_1+h_2) \cong \sO_{S}((3-k_2-k_3)\ell-k_1m)$, where $Y_1=(s_1)_0$.
\end{proof}

\section{Moduli of 't Hooft instantons and strata of special instantons}
In this section we describe the component of the moduli space of stable bundles with fixed Chern classes which corresponds to $h_i$-'t Hooft instantons. In particular we stratify the moduli space of instantons via the $h_i$-'t Hooft, the $(h_i,h_j)$-special  and the $h_i$-sectional special components.

{\textbf{Notation:}} Let us denote by $MI^D(c)$ the closed subscheme of $D$-'t Hooft instantons inside the Maruyama moduli space $M(0,c)$ of Gieseker stable, rank two vector bundle with $c_1=0$ and $c_2=c$. We will denote by $MI^{(D_i,D_j)}(c)$, with $i \neq j$, the space of $(D_i,D_j)$-special instantons and by $MI^{(D_i,D_i)}(c)$ the space of $D_i$-sectional special instantons.

Let us start by recalling that a $h_1$-'t Hooft bundle $\sE$ fits into \eqref{eh1tHooft}. Now consider $H\subset \mathrm{Hilb}$ the Hilbert scheme of locally complete intersection curves $C$ such that the canonical bundle is given by  $\omega_C\cong\sO_C(-2h_2-2h_3)$ and $h^0(\sO_C(-h_2-h_3))=0$. Consider $\mathfrak{H}\subset H\times X$ the universal curve and let $f$ be the natural projection $f:H\times X\to H$. We denote by $\mathfrak{I}_\fH$ the universal ideal sheaf, i.e. the sheaf on $H\times X$ whose restriction to $\{h\}\times X$ is the ideal sheaf $\sI_{C_h|X}$ where $h\in H$ and $C_h\subset X$ is the correspondent curve. Now consider the relative ext sheaf $\mathfrak{E}:=\sE xt_f^1(\mathfrak{I}(2h_1),\sO_{H\times X})$.

\begin{proposition}\label{pRelativeExt}
$\fE$ is a locally free sheaf on $H$.
\end{proposition}
\begin{proof}
We prove the statement by proving that each fiber of $\fE$ has constant dimension. Since $\fE$ is the first order derived functor of the composition of $f_\ast$ and $\sH om_{\sO_X}(\fI_\fH(2h_1),\bullet))$, the Grothendieck spectral sequence yields an exact sequence on low degree terms of the form
\begin{align*}
0 \to & R^1 f_\ast(\sH om(\fI_\fH(2h_1),\sO_{H\times X})) \to \fE \to f_\ast(\sE xt^1(\fI_\fH(2h_1),\sO_{H\times X})\to \\
\to & R^2 f_\ast(\sH om(\fI_\fH(2h_1),\sO_{H\times X}))
\end{align*}
Now let us fix $C$ which is isomorphic to a fiber of $f$ of the universal curve $\fH$. Notice that $\sH om(\fI_\fH(2h_1),\sO_{H\times X}))$ restricted to that fiber is just $\sH om (\sI_{C|X}(2h_1),\sO_X)$ which is isomorphic to $\sO_X(-2h_1)$. In particular $h^2(\sO_X(-2h_1))=0$, thus the base change isomorphism yields that each fiber of $R^2 f_\ast(\sH om(\fI_\fH(2h_1),\sO_{H\times X}))$ is zero, thus the same holds for the sheaf itself. The same argument implies that $R^1 f_\ast(\sH om(\fI_\fH(2h_1),\sO_{H\times X}))$ is a line bundle $\sA$ over $H$. We will finish the proof by showing that $f_\ast(\sE xt^1(\fI_\fH(2h_1),\sO_{H\times X})$ is a vector bundle, thus the same will hold for $\fE$ since $H$ is smooth. Using again base change we get 
\[
f_\ast(\sE xt^1(\fI_\fH(2h_1),\sO_{H\times X}) \otimes k(h) \cong H^0(X,\sE xt^1_{\{h\}\times X}(\fI_\fH(2h_1),\sO_X)).
\]
However $\sE xt^1_{|\{h\}\times X}(\fI_\fH(2h_1),\sO_X)\cong \sE xt^1(\sI_{C|X}(2h_1),\sO_X)$ where $C$ is the curve in $X$ corresponding to $h \in H$.  Apply now the contravariant functor $\sH om(\bullet,\sO_X)$ to the short exact sequence
\[
0\to \sI_{C|X}(2h_1) \to \sO_X(2h_1)\to \sO_C(2h_1)\to 0,
\]
so that the induced long exact sequence gives us $\sE xt^1(\sI_{C|X}(2h_1),\sO_X)\cong \sE xt^2(\sO_C(2h_1),\sO_X)$. But now
\[
\sE xt^2(\sO_C(2h_1),\sO_X) \cong \omega_C(2h_2+2h_3) \cong \sO_C
\]
thanks to Theorem \eqref{tCurves}. Finally the proof is complete by noticing that $h^0(\sO_C)$ is constant for each $C\in H$.
\end{proof}
Let us now consider the pullback diagram

\[\begin{tikzcd}
	\mathfrak{X} & H\times X \\
	\mathbb{P}(\fE^\vee) & H
	\arrow["\xi", from=1-1, to=1-2]
	\arrow["\zeta"', from=1-1, to=2-1]
	\arrow["f", from=1-2, to=2-2]
	\arrow["p"', from=2-1, to=2-2]
\end{tikzcd}\]
where $p$ is the natural map from the projective bundle to its base. Thanks to Proposition \ref{pRelativeExt}, there exists a universal family of extensions $\fF$ of $\xi^\ast \fI_\fH(2h_1)$ by $\xi^\ast \sO_{H\times X} \otimes \zeta^\ast \sO_{\pp(\fE^\vee)}$ over $\pp(\fE^\vee)$. Notice that $\fF$ is a sheaf on $\fX$ and it  represents a family of $h_1$-'t Hooft sheaves, since it may very well happen that there exists $C\in H$ such that one can find an element of $\Ext^1(\sI_{C|X}(2h_1),\sO_X)$ which is not locally free. Indeed one may take $C$ to be the disjoint union of two curves and one can take a section of $\sO_C$ which is identically zero on one of the components. Then the corresponding extension will not be locally free precisely on said component.

Nevertheless one can consider the restriction the open subset $\fP$ of $\pp(\fE^\vee)$ consisting images via $\zeta$ of elements for which $\fF$ is locally free. In this way each fiber of the universal extension is actually a $h_1$-'t Hooft bundle. In what follows we denote by $\sY$ the element $k_1e_1+k_2e_2+k_3e_3$ in $A^2(X)$. Thus we have a surjective morphism
\begin{equation}\label{eModuliMap}
\vartheta:\fP \to MI^{h_1}(\sY)
\end{equation}
from $\fP$ to the moduli space of $\mu$-stable, $h_1$-'t Hooft instanton bundles with fixed second Chern class.

In order to avoid special cases and unify the treatment, in the remainder of this section we will assume $k\ge 3$. The case $k=2$ will be discussed in the next section. Let us start by describing the smaller stratum of the moduli space of instantons.

\begin{theorem}\label{tModuliSectional}
The moduli space $MI^{(h_i,h_i)}(\sY)$ is smooth  and irreducible of dimension $2(k_j+k_w)+1$ with $i\neq j \neq w \neq i$ if and only if $k_i=0$. In case $k_i>0$, it is empty.
\end{theorem}

\begin{proof}
Thanks to Theorem \ref{tSerre}, Corollary \ref{cQuarticSectional} and Theorem \ref{tSpecialSectional}, it is enough to describe the variety of moduli ${\overline{MI}^{(h_i,h_i)}}(\sY)$ of $\mu$-stable instanton sheaves determined by taking a (non necessarily generating) section of $\wedge^2 (\sN_{Y|X} ) \otimes \sO_Y(-2h_i)$, where $Y$ is the zero locus of a section of $H^0(\sE(h_i))$. The space $MI^{(h_1,h_1)}(\sY)$ will be an open subset of ${\overline{MI}^{(h_i,h_i)}}(\sY)$.

Without loss of generality, we deal with the case $i=1$, the other case being completely symmetric.
The variety ${\overline{MI}^{(h_1,h_1)}}(\sY)$ is fibered over a variety $M$ by  
$$\mathbb{P}\Bigl(\bigl(H^0(\wedge^2 (\sN_{Y|X} ) \otimes \sO_F(-2h_1)_{|Y})\bigr)\oplus H^1(\sO_X(-2h_1))\Bigr) \cong \pp\bigl(\Ext^1(\sI_Y(2h_1),\sO_F)\bigr)\cong \mathbb{P}^{k_2+k_3}.$$
Now we describe $M$. It is fibered over $\pee2\cong \pp\bigl(H^0(\sO_X(2h_1))\bigr)$. Thus we have the following situation:
\[
{\overline{MI}^{(h_1,h_1)}}(\sY) \overset{\Gamma}{\twoheadrightarrow} M  \overset{\Psi}{\twoheadrightarrow} \pp^2
\]
The fibers of $\Psi$ are isomorphic to two disjoint copies $\pp^{k_2}\times \pp^{k_3}$, corresponding to a choice of one of the two rulings of one quadric. Thus $M$ is smooth since it is fibered over a smooth variety with smooth fibers. We now show that we can connect two different points in the disjoint fibers. The irreducibility will then follows from the connectedness and the smoothness. Suppose now that $k_2e_2$ is contained in a quadric which projects via $\pi_1$ to a point $p \in \mathbb{P}^1$, and $k_3e_3$ contained a quadric corresponding to $p'$. Connectedness then follows from the fact that we can algebraically connect $p$ and $p'$ without them colliding.

Finally, the dimension count is a direct consequence of the description of $M$ and the fact that $h^0(\sE(h_1))=2$, thus we have $\dim M=2+2k_2+2k_3-\dim\pp(H^0(\sE(h_1))$ and the proof is complete.
\end{proof}

\begin{theorem}\label{tModuliSpecial}
The moduli space $MI^{(h_i,h_j)}(\mathcal{Y})$ is irreducible and it is smooth outside of $\bigr(MI^{(h_i,h_i)} \cup MI^{(h_j,h_j)}\bigl)\cap MI^{(h_i,h_j)}$ with $i \neq j \neq w \neq i$ .Its dimension is given by
\begin{itemize}
\item [(i)] $2k+3$ if $\sY=ke_w$;
\item [(ii)] $k+3$ if $\sY=k_i(e_i+e_j)$;
\item [(iii)] $2k+2-\delta$ if $\sY=k_je_j+k_we_w$, with $\delta=1$ if $k_h=1$ and $\delta=0$ otherwise;
\item [(iv)] $2k+2-\delta$ if $\sY=k_ie_i+k_we_w$, with $\delta=1$ if $k_h=2$ and $\delta=0$ otherwise;
\item [(v)] $2k+2$ if $\sY=k_ie_i+k_je_j+e_j$.
\end{itemize}
In particular $MI^{(h_i,h_j)}(\mathcal{Y})$ is smooth in cases $(ii)$, $(iii)$, $(iv)$ and $(v)$.
\end{theorem}
\begin{proof}
First of all, observe that the cases $(i)-(v)$ are exactly the one appearing in Theorem \ref{tSpecial}. Arguing as in the proof of Theorem \ref{tModuliSectional}, we describe the variety of moduli ${\overline{MI}^{(h_i,h_j)}}(\sY)$ of $\mu$-stable, $(h_i,h_j)$-special instanton sheaves determined by taking a (non necessarily generating) section of $\wedge^2 (\sN_{Y|X} ) \otimes \sO_Y(-2h_i)$, where $Y$ is the zero locus of a section of $H^0(\sE(h_i))$. The space $MI^{(h_i,h_j)}(\sY)$ will be an open subset of ${\overline{MI}^{(h_i,h_j)}}(\sY)$. In order to simplify the notation, we assume $i=1$ and $j=2$, the other two cases being completely symmetric.

The variety ${\overline{MI}^{(h_1,h_2)}}(\sY)$ is fibered over a variety $M$ by  
$\mathbb{P}^{k_2+k_3}$, thus in what follows we describe $M$ in the different cases appearing in the statement.

{\textbf{Case (i)}:}\\
This is the only case where the dependency locus can be either reducible or irreducible. $M$ is fibered over $\pp^3 \cong \pp \bigl(H^0(\sO_X(h_1+h_2))\bigr)$ corresponding to irreducible surfaces. Notice that on the one hand the fiber over a smooth $S$ is isomorphic to $\pp^k=\pp\bigl(H^0(\sO_S(k\ell))\bigr)$ which corresponds to a choice of $k$ lines (counted with multiplicity) in the $\ell$-ruling of $S$. On the other hand, fibers over reducible surfaces are isomorphic to 
\[
\bigcup_{k'+k''=k}\pp\bigl(H^0(\sO_{Q_1}(k' m))\bigr)\times \pp\bigl(H^0(\sO_{Q_2}(k'' m))\bigr)
\]where $Q_i \in |h_i|$. Thus $M$ is smooth, being a fibration with smooth fibers over a smooth base. We now prove that it is irreducible. Consider the projection $\pi_{12}:X \to Q:=\pp(V_1)\times \pp(V_2)$. Each line in the family $e_3$ projects through $\pi_{12}$ to a point of $Q$ and each surface $S$ projects to a conic in $\pp^1\times \pp^1$. Thus $M$ can be described as the incidence variety
\[
M=\{(Z,C)\ | \ Z \ \text{is a $0$-dimensional scheme in $\pp^1 \times \pp^1$,} \ C \in |\ell+m| \ \text{and} \ Z\subset C\}\subset Q\times \pp^3
\]
where $\ell$ and $m$ are the generators of $\Pic(\pp^1 \times \pp^1)$. $M$ comes with a natural projection $\tau$ over $\pp^3$. Let $\sD$ be the divisor in $\pp^3$ corresponding to reducible conics in $|\ell+m|$ and $\sU$ the open subset $\pp^3 \setminus \sD$. Now $\tau^{-1}(\sU)=\{(Z,C)\ | \ \text{$C$ is irreducible}\}$ is open and irreducible in $M$, since $\sU$ is irreducible and the same is true for any fiber over points of $\sU$. Notice now that any singular conic in $Q$ is nodal, thus it is smoothable.
Moreover every zero–dimensional subscheme on it deforms along a smoothing (see for example \cite[Section 14]{Har3}). The irreducibility of the total space then follows from the fact that any point of $M \setminus \tau^{-1}(\sU)$ can be obtained as a specialization of $(Z,C)$ with $C$ irreducible.  

{\textbf{Case (ii)}:}\\
In this case $S$ is irreducible and $M$ is fibered over an open subset of $\pp^3$ by $\pp^k\cong \pp\bigl(H^0(\sO_S(k\ell)))$ thus it is smooth and irreducible.

{\textbf{Case (iii)}:}\\
This is the first totally reducible case. Here $M$ is fibered over $$\pp^1 \times \pp^1\cong \pp\bigr(H^0(\sO_X(h_1))\bigl)\times \pp\bigr(H^0(\sO_X(h_2))\bigl)$$. The fiber over each point is isomorphic to
$$
\pp^{k_2}\times \pp^{k_3}\cong \pp\bigl(H^0(\sO_{Q_1}(k_2m))\bigr) \times \pp\bigl(H^0(\sO_{Q_2}(k_3\ell))),
$$
so $M$ is smooth and irreducible.

{\textbf{Case (iv)}:}\\
Here $M$ is fibered over $\pp^1 \times \pp^1$ with fibers isomorphic to
$$
\pp^{2k_1+1} \times \pp^{k_3-1}\cong \pp\bigl(H^0(\sO_{Q_2}(k_1\ell+m))\bigr) \times \pp \bigl(H^0(\sO_{Q_1}((k_3-1))\ell)\bigr)
$$
and again $M$ is smooth and irreducible.

{\textbf{Case (v)}:}\\
This case is analogous to the previous one.

Finally, the presence of $\delta$ depends on the fact that in case $(iii)$ (resp. $(iv)$) if $k_3=1$ (resp. $k_3=2$) then $Y$ is contained in a reducible quartic in the linear system $|2h_1|$ (resp. $2h_2$). Thanks to Corollary \ref{cQuarticSectional}, this is equivalent to the fact that the instanton bundle $\sE$ associated to $Y$ by means of the Serre correspondence satisfies $h^0(\sE(h_1))=2$ (resp. $h^0(\sE(h_2))=2$). The dimensional count then follows directly from the above description.
\end{proof}

\begin{theorem}\label{tModulitHooft}
The moduli space $MI^{h_i}(\sY)$ consists of $\Delta(\sY)$ (cf. \eqref{eNumberComponents}) irreducible components of dimension $2k_1+3k_2+3k_3$. It is smooth outside the locus $MI^{(h_i,h_i)}(\sY)$. The union $MI^{h_1}(\sY) \cup MI^{h_2}(\sY) \cup MI^{h3}(\sY)$ is smooth in the open locus of points not lying in the closed variety $\bigcup_{i} MI^{(h_i,h_i)}(\sY)\cup \bigcup_{i<j} MI^{(h_i,h_j)}(\sY)$.
\end{theorem}
\begin{proof}
As usual, we deal with the case $i=1$. Arguing as in the proofs of Theorem \ref{tModuliSectional} and \ref{tModuliSpecial}, we describe ${\overline{MI}^{h_1}}(\sY)$ of $\mu$-stable instanton sheaves determined by taking a (non necessarily generating) section of $\wedge^2 (\sN_{Y|X} ) \otimes \sO_Y(-2h_1)$, where $Y$ is the zero locus of a section of $H^0(\sE(h_1))$.

The space $MI^{h_1}(\sY)$ will be an open subset of ${\overline{MI}^{h_1}}(\sY)$. As an additional reduction, we describe the locus ${MI'}^{h_1}(\sY)$ of $MI^{h_1}(\sY)$, which consists of $h_1$-'t Hooft bundles such that $h^0(\sE(h_1))=1$. This corresponds, by means of Corollary \ref{cQuarticSectional}, to points in $H$ representing curves not lying on a surface in the linear system $|2h_1|$. Let us denote by $\sA\subset H$ this open subset. The restriction $\mathfrak{P}_{|\sA}$ is then isomorphic to ${MI'}^{h_1}(\sY)$ via the map $\vartheta$ \eqref{eModuliMap}. Thus ${MI'}^{h_1}(\sY)$ is smooth, since $\mathfrak{P}_{|\sA}$ is an open subset of the smooth variety $\mathfrak{P}$. In particular ${MI'}^{h_1}(\sY)$ is fibered over $\sA$ by $\pp^{k_2} \times \pp^{k_3}$, thus the part of the statement regarding the number of irreducible components follows from Proposition \ref{pHilbertComponents}.
\end{proof}

We end this section by recalling a result regarding the whole moduli space $MI(\sY)$ of $\mu$-stable instanton bundles with $c_2(\sE)=k_1e_1+k_2e_2+k_3e_3$ for arbitrary triples $(k_1,k_2,k_3)$.

\begin{theorem}\cite[Theorem 1.2]{AnMa}\label{tInstaModuli}
There exists inside $MI(\sY)$ a generically smooth, irreducible component of dimension $4k-3$.
\end{theorem}

We conclude this section by noticing that as soon as $k \ge 4$ then $4k-3$ is strictly bigger than $2k_1+3k_2+3k_3$, thus $MI^{h_i}(k_1e_1+k_2e_2+k_3e_3)$ is strictly contained in $MI(k_1e_1+k_2e_2+k_3e_3)$ and it cannot cover the whole moduli spaces of instanton bundles. In the next section we will deal with the case $k=2$ and $k=3$.

\section{The low charge cases}
In this final section, we deal specifically with the charge 2 and 3 cases. In particular the charge $3$ is already very explicative of the possible configurations of special, sectional special and 't Hooft bundles.

As pointed out in the previous sections, charge 2 instanton bundles, after a twist by $\sO_X(h)$ are actually Ulrich bundles. The moduli spaces of rank two Ulrich bundles on $X$ have been studied in \cite{CFM2}. In particular we recall here the following result (slightly adapted to our notation):
\begin{theorem}\cite[Theorem B (3)]{CFM2}
    The moduli space $MI(2e_3)$ of indecomposable, initialized, semistable instanton bundles is generically smooth and rational of dimension 5:
its general point corresponds to a stable bundle and it also contains exactly one point representing the equivalence class of all the strictly semistable instanton bundles with $c_1=0$.

The moduli space $MI(e_2+e_3)$ is smooth and unirational of dimension 5: its points
correspond to stable bundles.
\end{theorem}
In the following table we collect the values $\chi(\sE_i)$ for a charge two instanton bundle. Notice that $\chi(\sE(h_i))=h^0(\sE(h_i))-h^1(\sE(h_i))$ since the vanishing of $h^2$ and $h^3$ follow from the definition of instanton. Thus as soon as $\chi(\sE(h_i))>0$ we get that $\sE$ is a $h_i$-'t Hooft bundle.
\begin{footnotesize}
\begin{table}[H]
\label{TableCharge2}
\centering
\begin{tabular}{|c|c|c|c|}
\hline
\rule[-4mm]{0mm}{1cm}
$c_2(\sE)$ & $\chi(\sE(h_1))$ & $\chi(\sE(h_2))$ & $\chi(\sE(h_3))$ \\
\hline
\rule[-4mm]{0mm}{1cm}
$e_1+e_2$ &  \boxed{1} & \boxed{1} & \boxed{2} \\
\hline
\rule[-4mm]{0mm}{1cm}
$e_2+e_3$ & \boxed{2} & \boxed{1} & \boxed{1} \\
\hline
\rule[-4mm]{0mm}{1cm}
$e_1+e_3$ & \boxed{1} & \boxed{2} & \boxed{1} \\
\hline
\rule[-4mm]{0mm}{1cm}
$2e_2$ & \boxed{2} & 0 & \boxed{2} \\
\hline
\rule[-4mm]{0mm}{1cm}
$2e_3$ & \boxed{2} & \boxed{2} & 0 \\
\hline
\end{tabular}
\caption{The charge $2$ (Ulrich) case}
\end{table}
\end{footnotesize}
We observe that every charge two instanton bundle is $h_i$-'t Hooft for some $i$ and in the first three cases $h^0(\sE(h_i))>0$ for all $i$. We already noticed and characterized this behavior in Proposition \ref{pVerySpecial}. In the minimal charge case, all the definitions of $h_i$-'t Hooft, $(h_i,h_j)$-special and $h_i$-sectional special somehow collapse, giving a unique, generically smooth and irreducible component.

Once we increase the charge, the landscape becomes much richer, allowing different and interesting configurations of the moduli spaces. Let us start by compiling the table of the Euler characteristics in the charge three case.
\begin{footnotesize}
\begin{table}[H]
\label{TableCharge3}
\centering
\begin{tabular}{|c|c|c|c|c|}
\hline
\rule[-4mm]{0mm}{1cm}
 & $c_2(\sE)$ & $\chi(\sE(h_1))$ & $\chi(\sE(h_2))$ & $\chi(\sE(h_3))$ \\
\hline
\rule[-4mm]{0mm}{1cm}
(a) & $e_2+2e_3$ &  \boxed{1} & 0 & -1 \\
\hline
\rule[-4mm]{0mm}{1cm}
(b) & $2e_2+e_3$ & \boxed{1} & -1 & 0 \\
\hline
\rule[-4mm]{0mm}{1cm}
(c) & $3e_2$ & \boxed{1} & -2 & \boxed{1} \\
\hline
\rule[-4mm]{0mm}{1cm}
(d) & $3e_3$ & \boxed{1} & \boxed{1} & -2 \\
\hline
\rule[-4mm]{0mm}{1cm}
(e) & $2e_1+e_2$ & -1 & 0 & \boxed{1} \\
\hline
\rule[-4mm]{0mm}{1cm}
(f) & $2e_1+e_3$ & -1 & \boxed{1} & 0 \\
\hline
\rule[-4mm]{0mm}{1cm}
(g) & $e_1+2e_2$ &  0 & -1 & \boxed{1} \\
\hline
\rule[-4mm]{0mm}{1cm}
(h) & $e_1+2e_3$ &  0 & \boxed{1} & -1 \\
\hline
\rule[-4mm]{0mm}{1cm}
(i) & $e_1+e_2+e_3$ &  0 & 0 & 0 \\
\hline
\end{tabular}
\caption{The charge $3$ case}
\end{table}
\end{footnotesize}
From the table we see that, except for the case $c_2(\sE)=e_1+e_2+e_3$, there always exists an $i$ such that $h^0(\sE(h_i))>0$. We only deal with cases $(a)$, $(d)$ and $(i)$, since by permuting indices we get a description of all the other cases. In all cases, we will make a systematic use of Theorem \ref{tModuliSectional}, \ref{tModuliSpecial}, \ref{tModulitHooft} and \ref{tInstaModuli}.

{\textbf{Case (a)}:}\\
In this case $MI^{h_1}(e_2+2e_3)$ covers the entire moduli space $MI(e_2+2e_3)$ of rank two instanton bundles having $c_2(\sE)=e_2+2e_3$. Since $k_1=0$, Theorem \ref{tModulitHooft} yields that $MI^{h_1}(e_2+2e_3)$ is irreducible and smooth outside the closed locus $MI^{(h_1,h_1)}(e_2+2e_3)$. Geometrically, each $h_1$-'t Hooft bundle corresponds to a curve $Y\in H$ representing $e_2+2e_3$ in $A^2(X)$ whose reduced structure is the disjoint union of lines, thus points in $MI^{(h_1,h_1)}(e_2+2e_3)$ projects to curves $Y$ such that the component $2e_3$ is contained in a quadric $Q \in |h_1|$. In particular $MI^{h_1}(e_2+2e_3)$ has dimension $9$ and it is smooth outside the two codimensional locus $MI^{(h_1,h_1)}(e_2+2e_3)$. Moreover, we can expect from Table \ref{TableCharge3} that not all charge $3$ istanton bundles are $h_2$-'t Hooft or $h_3$-'t Hooft. $MI^{h_2}(e_2+2e_3)$ and $MI^{h_3}(e_2+2e_3)$ are smooth and irreducible of dimensions $8$ and $7$, respectively. They are both contained in $MI^{h_1}(e_2+2e_3)$, thus they actually coincide with $MI^{(h_1,h_2)}(e_2+2e_3)$ and $MI^{(h_1,h_3)}(e_2+2e_3)$. We claim that $MI^{(h_1,h_3)}(e_2+2e_3)$ coincides with $MI^{(h_1,h_1)}(e_2+2e_3)$. Indeed, take $\sE$ a point in $MI^{(h_1,h_3)}(e_2+2e_3)$ and let $Y_{\sE}$ be the correspondent element in $H$. Thanks to Theorem \ref{tSpecial}, $Y_{\sE}$ is contained in a reducible union of two quadrics $Q_1\in |h_1|$ and $Q_3\in |h_3|$. In particular the component representing $2e_3$ is contained in $Q_1$ and the line representing $e_2$ lies in $Q_3$. However, every $e_2$ line is also contained in a quadric $Q'_1$ in the linear system $|h_1|$, thus $Y_{\sE}$ is contained in a reducible $Q_1 \cup Q'_1$. Finally we point out that the space $MI^{(h_2,h_3)}(e_2+2e_3)$ is empty thanks to Theorem \ref{tSpecial}.

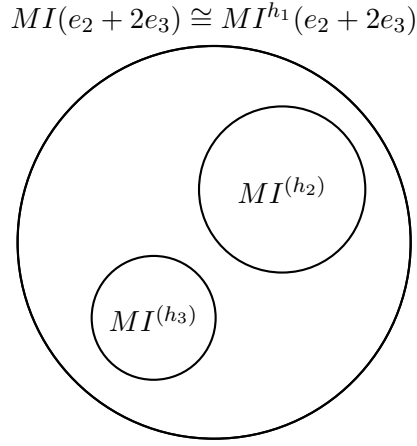
\bigskip

\begin{center}
\begin{tikzpicture}
        \def\firstcircle{(0,0) circle (2.6cm)}
        \def\secondcircle{(0.9,0.7) circle (1.1cm)}
        \def\thirdcircle{(-0.8,-1) circle (0.82cm)}

        \colorlet{circle edge}{black!100}
        \colorlet{circle area}{teal!0}
        
        \tikzset{filled/.style={fill=circle area, draw=circle edge, thick},
            outline/.style={draw=circle edge, thick}}
        
        \setlength{\parskip}{5mm}
        \begin{scope}
            \fill[filled] \firstcircle;
            \clip \secondcircle;
            \clip \thirdcircle;
        \end{scope}
        \draw[outline] \firstcircle node {};
        \draw[outline,fill=magenta!0] \secondcircle node {$MI^{(h_2)}$};
        \draw[outline,fill=violet!0] \thirdcircle node {$MI^{(h_3)}$};
        %\draw[outline,fill=cyan!30] \fourthcircle node {$MI^{(h_1,h_1)}$};
        \node[anchor=south] at (current bounding box.north) {${MI(e_2+2e_3)\cong MI^{h_1}(e_2+2e_3)}$};
    \end{tikzpicture}
\end{center}

{\textbf{Case (d)}:}\\
In this case both $MI^{h_1}(3e_3)$ and $MI^{h_2}(3e_3)$ cover the whole moduli space $MI(3e_3)$ of rank two instanton bundles having $c_2(\sE)=3e_3$. In other word each instanton bundle corresponding to a point in $MI(3e_3)$ is $(h_1,h_2)$ special. Geometrically this is equivalent to the fact that every curve $Y$ in $H$ representing $3e_3$ is always contained in a quartic surface $S \in |h_1+h_2|$. To see this, consider the $\pi_{12}$ projection of $X$ to the quadric $Q:=\pp(V_1) \times \pp(V_2)$. As usual, let us denote by $\ell$ and $m$ the generators of $\Pic(Q)$. Then $S$ is the pullback via $\pi_{12}$ of a conic on $Q$, namely a divisor $\ell+m$. The curve $Y$ project to $Q$ to a $0$-dimensional scheme $Z$ of length three. But now, given $Z$ we can always find a conic passing through $Z$, which correspond on $X$ to a surface $S$ containing $Y$. Thus $MI(3e_3)$ is irreducible and it is smooth outside the two codimensional locus $MI^{(h_1,h_1)}(3e_3) \cup  MI^{(h_2,h_2)}(3e_3)$. We claim that $MI^{(h_1,h_1)}(3e_3) \cap  MI^{(h_2,h_2)}(3e_3)=\emptyset$. Notice that this is a consequence of Proposition \ref{pDoubleSpecial}, nevertheless we highlight here the geometric picture. Consider $Y_1=(s_1)_0$ with $s_1 \in H^0(\sE(h_1))$. Suppose $\sE\in MI^{(h_1,h_1)}(3e_3)$ and let $Z_1$ be the projection of $Y_1$ via $\pi_{12}$. We deal with the case of $Z_1$ being reduced. Corollary \ref{cQuarticSectional} implies that $Z_1$ is contained in the union of two fibers of $Q$, say $|2\ell|$. In particular, two of the three points of $Z_1$ are contained in one line $L$ and the other belongs to a different disjoint line $L'$. These two lines correspond to two quadric surfaces in $|h_1|$. Now Sequence \eqref{eSpecial12} implies that $\sE$ is $h_2$-sectional special if and only if $h^0(\sI_{Y_1}(h_1+h_2))=2$. As we previously noted, this is equivalent to require that a $\pp^1$ of conics in $Q$ passes through $Z_1$. However each smooth conic $C$ intersects $L$ in just one point, thus $Z_1\not\subset C$. Moreover the only singular conic containing $Z_1$ is the union of $L$ and $M$ where $M$ is the unique line in $|m|$ passing through the point of $Z_1$ contained in $L'$. Thus $h^0(\sI_{Y_1}(h_1+h_2))=1$ and $\sE$ is not $h_2$-special.

Finally, notice that in case $(d)$ we obtain that $MI^{h_3}(3e_3)$ is empty, since for any curve $Y$ representing $3e_3$ we have $\omega_Y \not \cong \sO_Y(-2h_1-2h_2)$.

\bigskip

\begin{center}
\begin{tikzpicture}
        \def\firstcircle{(0,0) circle (2.6cm)}
        \def\secondcircle{(0,1.2) circle (0.9cm)}
        \def\thirdcircle{(0,-1.2) circle (0.9cm)}

        \colorlet{circle edge}{black!100}
        \colorlet{circle area}{teal!0}
        
        \tikzset{filled/.style={fill=circle area, draw=circle edge, thick},
            outline/.style={draw=circle edge, thick}}
        
        \setlength{\parskip}{5mm}
        \begin{scope}
            \fill[filled] \firstcircle;
            \clip \secondcircle;
            \clip \thirdcircle;
        \end{scope}
        \draw[outline] \firstcircle node {};
        \draw[outline,fill=yellow!0] \secondcircle node {$MI^{(h_1,h_1)}$};
        \draw[outline,fill=yellow!0] \thirdcircle node {$MI^{(h_2,h_2)}$};
        %\draw[outline,fill=red!20] \fourthcircle node {$MI^{(h_1,h_3)}$};
        \node[anchor=south] at (current bounding box.north) {${MI(3e_3)\cong MI^{h_1}(3e_3) \cong MI^{h_2}(3e_3)}$};
    \end{tikzpicture}
\end{center}

{\textbf{Case (i)}:}\\
This is the only case were we cannot expect any moduli space of $(h_i)$-'t Hooft bundles to cover the whole moduli space $MI(e_1+e_2+e_3)$. Indeed there exists in $MI(e_1+e_2+e_3)$ a generically smooth, irreducible component $\hat{MI}(e_1+e_2+e_3)$ of dimension $9$. The loci $MI^{h_i}(e_1+e_2+e_3)$ have codimension one in $MI(e_1+e_2+e_3)$. Each $MI^{h_i}(e_1+e_2+e_3)$ is smooth (i.e. there are no $h_i$-sectional special instantons) and consists of two irreducible components. The moduli spaces $MI^{(h_i,h_j)}(e_1+e_2+e_3)$ are smooth and irreducible of dimension $8$ and their (disjoint) union actually coincide with $\bigsqcup_{i}MI^{h_i}(e_1+e_2+e_3)$, i.e. each irreducible component of $MI^{h_i}(e_1+e_2+e_3)$ consists either of $(h_i,h_j)$-special or $(h_i,h_w)$-special instanton bundles, with $i \neq j \neq w \neq i$. In other words we have $MI^{h_i}(e_1+e_2+e_3)=MI^{(h_i,h_j)}(e_1+e_2+e_3)\sqcup MI^{(h_i,h_w)}(e_1+e_2+e_3)$.

\bigskip

\begin{center}
\begin{tikzpicture}
        \def\firstcircle{(0,0) circle (2.6cm)}
        \def\secondcircle{(1,-1) circle (0.9cm)}
        \def\thirdcircle{(-1,-1) circle (0.9cm)}
        \def\fourthcircle{(0,1) circle (0.9cm)}
        
        \colorlet{circle edge}{black!100}
        \colorlet{circle area}{teal!0}
        
        \tikzset{filled/.style={fill=circle area, draw=circle edge, thick},
            outline/.style={draw=circle edge, thick}}
        
        \setlength{\parskip}{5mm}
        \begin{scope}
            \fill[filled] \firstcircle;
            \clip \secondcircle;
            \clip \thirdcircle;
        \end{scope}
        \draw[outline] \firstcircle node {};
        \draw[outline,fill=red!0] \secondcircle node {$MI^{(h_1,h_2)}$};
        \draw[outline,fill=red!0] \thirdcircle node {$MI^{(h_2,h_3)}$};
        \draw[outline,fill=red!0] \fourthcircle node {$MI^{(h_1,h_3)}$};
        \node[anchor=south] at (current bounding box.north) {${\hat{MI}(e_1+e_2+e_3)}$};
    \end{tikzpicture}
\end{center}
\begin{remark}
We conclude the work by remarking that the charge $3$ case is the last one where the two notions of speciality are strongly related and coincide in some cases. As soon as $k\ge 4$, in light of Proposition \ref{pDoubleSpecial} and Remark \ref{rDoubleSpecial}, the moduli spaces of special and sectional special bundles never coincide, even though it is possible for them to have non-empty intersection.
\end{remark}
\bibliographystyle{amsplain}

\begin{thebibliography}{99}
\bibitem{AnCa}
V. Antonelli, G. Casnati: \emph{Instanton sheaves on projective schemes}, J. Pure Appl. Alg. \textbf{227} (2023).

\bibitem{AnMa2}
V. Antonelli, F. Malaspina: {\em $H$--instanton bundles on three-dimensional polarized projective varieties}, J. Algebra \textbf{598}, 570--607 (2022).

\bibitem{AnMa}
V. Antonelli, F. Malaspina: {\em Instanton bundles on the Segre threefold with Picard number three}, Math. Nach. \textbf{293}, 1026--1043 (2020).

\bibitem{AMMP}
V. Antonelli, F. Malaspina, S. Marchesi, J. Pons-Llopis: \emph{'t Hooft bundles on the complete flag threefold and moduli spaces of instantons}, J. Math. Pures Appl. \textbf{202} (2025).

\bibitem{Ar}
E. Arrondo: \emph{A home--made Hartshorne--Serre correspondence}. Rev Mat. Complut.  \textbf{ 20}, 423--443 \rm (2007).  

\bibitem{ADHM}
M. F. Atiyah, V.G. Drinfeld, N. J. Hitchin, Yu. I Manin: \emph{Construction of instantons}, Phys. Lett. A, \textbf{65}, 185--187 (1978).

\bibitem{AW}
M.F. Atiyah, R. S. Ward: \emph{Instantons and Algebraic Geometry}. Commun. math. Phys., \textbf{55}, 117--124 (1977).

\bibitem{BF}
C. Banica, O. Forster: \emph{Multiplicity Structures on Space Curves}, Algebraic Geometry, Proc. Lefschetz Centen. Conf., Mexico City/ Mex. 1984, Part I, Contemp. Math., \textbf{58}, 47--64 (1986).

\bibitem{BeFr}
V. Beorchia, D. Franco: \emph{On the moduli space of ’t Hooft bundles}, Ann. Univ. Ferrara, \textbf{47}, 253--268 (2001).

\bibitem{BT}
W. B\"ohmer, G. Trautmann: \emph{Special instanton bundles and Poncelét curves}, Proc. Lambrecht Conf., 325–336, (Lect. Notes Math., Vol. 1273). Berlin Heidelberg New York: Springer (1987)

%\bibitem{CFM}
%G. Casnati, D. Faenzi, F. Malaspina: \emph{Rank two aCM bundles on the del Pezzo threefold with Picard number 3}, J. Algebra, \textbf{429}, 413--446 (2015).

\bibitem{CFM2}
G. Casnati, D. Faenzi, F. Malaspina: \emph{Moduli spaces of rank two aCM bundles on the Segre product of three projective lines}, J. Pure Appl. Algebra, \textbf{220}, 1554--1575 (2016).

%\bibitem{CJMM}
%G. Comaschi, M. Jardim, C. Martinez, D. Mu: \emph{Instantons: the next frontier}, S\~{a}o Paulo J. Math. Sci. (2023).

\bibitem{CMRP}
L. Costa, R. M. Mir\'o-Roig, J. Pons-Llopis: \emph{Ulrich bundles}, De Gruyter Studies in Mathematics, \textbf{77}, De Gruyter
(2021).

\bibitem{CO}
L. Costa, G. Ottaviani: \emph{Nondegenerate multidimensional matrices and instanton bundles}, Trans. Amer. Math. Soc., \textbf{355} (1), 49--55 (2003).

\bibitem{E-S-W}
D. Eisenbud, F.O. Schreyer, J. Weyman: \emph{Resultants and Chow forms via exterior syzigies}.  J. Amer. Math. Soc. \textbf{16}, 537--579 (2003).

\bibitem{Fa2}
D. Faenzi: \emph{Even and odd instanton bundles on Fano threefolds of Picard number one},  Manuscripta Math., \textbf{144}, 199--239 (2014).

\bibitem{Har}
R. Hartshorne: {\em Algebraic geometry}. G.T.M. 52, Springer \rm (1977).

\bibitem{Har3}
R. Hartshorne: {\em Deformation Theory}. G.T.M. 257, Springer \rm (2010).

\bibitem{Ha2}
R. Hartshorne: \emph{Stable vector bundles of rank $2$ on $\mathbb{P}^3$}, Math. Ann., \textbf{238}, 229--280 (1978).

\bibitem{HN}
A. Hirschowitz, M. S. Narasimhan: \emph{Fibr\'es de 't Hooft Sp\'eciaux et Applications}, in: Le Barz, P., Hervier, Y. (eds) Enumerative Geometry and Classical Algebraic Geometry. Progress in Mathematics, \textbf{24}. Birkh\"auser Boston (1982).

\bibitem{Hooft}
G. 't Hooft: \emph{Computation of the quantum effects due to a four-dimensional pseudoparticle}, Phys. Rev. D, \textbf{14} (12), 3432--3450 (1976)

\bibitem{JMPS}
M. Jardim, G. Menet, D. Prata, H. S\'a Earp: \emph{Holomorphic bundles for higher dimensional gauge theory}, Bull. London Math. Soc., \textbf{49}, 117--132 (2017).

\bibitem{JV}
M. Jardim, M. Verbitsky: \emph{Trihyperk\"{a}hler reduction and instanton bundles on $\mathbb{C}\pp^3$}, Compositio Mathematica, \textbf{150} (11), 1836--1868 (2014).

\bibitem{Kuz}
A. Kuznetsov: \emph{Instanton bundles on {Fano} threefolds}, Central European Journal of Mathematics, \textbf{4}, 1198--1231 (2012).

\bibitem{Ma}
N. Manolache: \emph{Multiple structure on smooth support}, Math. Nach. \textbf{167}, 157--202 (1994).

\bibitem{MMP}
F. Malaspina, S. Marchesi, J. Pons-Llopis: {\em Instanton bundles on the flag variety $F(0,1,2)$}.  Ann. Sc. Norm. Super. Pisa Cl. Sci. (5) \textbf{20} (2020), 1469--1505.

%\bibitem{O-S-S}
%C. Okonek, M. Schneider, H. Spindler: {\em Vector bundles on complex projective spaces}, Progress in Mathematics 3, Birkh\"auser \rm(1980).

\bibitem{T1}
A. S. Tikhomirov: \emph{Moduli of mathematical instanton vector bundles with odd c2 on projective space},
Izv. Math., \textbf{76} (5), 143--224 (2012).

\bibitem{T2}
A. S. Tikhomirov: \emph{Moduli of mathematical instanton vector bundles with even c2 on projective space},
Izv. Math., \textbf{77} (6), 1195--1223 (2013).

\end{thebibliography}

\end{document}